\documentclass[review,3p,times,sort&compress]{elsarticle}
\usepackage{amssymb}
\usepackage{graphicx,fancyhdr}
\usepackage{multirow}
\usepackage{dsfont}
\usepackage{subfigure}
\usepackage{bbm}
\usepackage{booktabs}
\usepackage{amsthm}
\usepackage{bbm,bm}
\DeclareMathAlphabet{\mathpzc}{OT1}{pzc}{m}{it}
\usepackage{bbm}
\usepackage{dsfont}
\usepackage{yfonts}
\usepackage{amsmath}
\usepackage{amssymb}
\usepackage{amsfonts}
\usepackage{mathrsfs,float}
\usepackage{color}
\usepackage[colorlinks, citecolor=blue]{hyperref}

\makeatletter

\journal{}

\begin{document}
	
	\begin{frontmatter}
		
		
		\newtheorem{thm}{Theorem}[section]
		\newtheorem{rem}[thm]{Remark}
		\newtheorem{ex}{Example}
		\newtheorem{case}{Case}
		\newtheorem{pro}[thm]{Proposition}
		\newtheorem{defi}[thm]{Definition}
		\newtheorem{ass}{Assumption}
		\newtheorem{lem}[thm]{Lemma}
	\newtheorem{corollary}[thm]{Corollary}
\newtheorem{assumption}[thm]{Assumption}
		\newproof{pf}{Proof}
		\newproof{pot}{Proof of Theorem \ref{thm2}}
		\title{High-precision quadrature via local Fourier extension: analytic integration, uniform sampling, and correction for piecewise smooth integrands}
		
		\author[1]{Xinran Liu}
		\author[1]{Zhenyu Zhao}
\author[1]{Benxue Gong}
		\address[1]{School of Mathematics and Statistics, Shandong University of Technology, Zibo, 255049, China}
	\begin{abstract}
We propose a high-precision numerical quadrature framework based on local Fourier extension (LFE) approximations. The method constructs, on each subinterval, a truncated-SVD stabilized local Fourier continuation of the integrand on an extended periodic domain, and then evaluates the integral \emph{analytically} from the resulting Fourier coefficients. Under uniform sampling, the discrete LFE matrix and its TSVD factors are precomputed once and reused across all windows, yielding an efficient offline/online implementation that remains compatible with classical composite rules.

We provide an error bound that reduces the quadrature error to the LFE approximation error and derive algebraic convergence rates for Sobolev-regular integrands. Numerical experiments demonstrate that, on smooth functions, the proposed quadrature reaches near machine precision with substantially fewer nodes than the composite Simpson rule. The advantage persists for oscillatory and variable-frequency integrands and becomes more pronounced for nonuniform phase structures.

For continuous piecewise smooth integrands, we develop a correction strategy driven by coefficient-energy outliers to identify singularity-containing windows, followed by a localized procedure that brackets the singular point within one grid cell and corrects only the affected window contribution. The corrected quadrature restores near-spectral accuracy in the reported tests, including cases where the singularity is not aligned with the window endpoints.
\end{abstract}
		
		
	\end{frontmatter}

\section{Introduction}

Numerical integration is a fundamental task in scientific computing and mathematical modeling, arising whenever integrals must be evaluated from analytic expressions, numerical solutions, or discrete data \cite{Yang2005}.
It plays a central role in applications ranging from differential equations and inverse problems to signal processing, data analysis, and computational physics \cite{dong2020stability,kong2021application,spiteri2020efficient}.
In many realistic settings, integrals cannot be computed in closed form or to sufficient accuracy by direct evaluation, making efficient and reliable numerical quadrature indispensable.

Most classical quadrature rules are built on polynomial or piecewise-polynomial approximation of the integrand, followed by analytic integration of the approximant.
Representative examples include Newton--Cotes formulas, Gaussian quadrature, and spline-based schemes \cite{Busenberg1986,magalhaes2021new,ciobanu2022fast}.
These methods are highly effective when the integrand is smooth and moderately varying.
However, their performance can deteriorate for highly oscillatory, non-periodic, or locally structured integrands, where polynomial-based approximations often demand dense sampling or high-order constructions, increasing computational cost and reducing robustness.
Moreover, although Gaussian quadrature can achieve very high algebraic accuracy, it relies on specially designed nonuniform nodes, whereas in many applications the available data are given only on uniformly spaced grids (e.g., from measurements or from PDE solvers with prescribed spatial discretizations), so Gaussian-type rules cannot be applied directly without additional interpolation or resampling.

A related line of work aims at improving computational efficiency through fast algorithms and adaptive strategies \cite{rao2014adjoint,ciobanu2022fast}.
These approaches can reduce arithmetic complexity and often improve practical performance, but they typically retain the same underlying approximation paradigm and may still face difficulties when the integrand exhibits complex or spatially varying frequency content.

Spectral representations offer an alternative viewpoint.
Fourier-based approximations are naturally suited for oscillatory functions and can deliver very high accuracy with relatively few degrees of freedom \cite{luchini1994fourier,iserles2004numerical}.
Classical Fourier methods, however, rely on periodicity and may suffer from boundary artifacts on non-periodic intervals.
Fourier extension techniques address this limitation by embedding a non-periodic function into a larger periodic domain, enabling Fourier-type spectral approximation without requiring periodicity \cite{huybrechs2010fourier,matthysen2018function,zhao2025fast}.
These ideas are closely related to high-accuracy quadrature on nonuniform nodes: for example, Clenshaw--Curtis quadrature is a classical spectral-type scheme, and its convergence behavior for functions of limited regularity has been studied in detail \cite{xiang2012convergence}.

More recently, \emph{localized} Fourier extension methods have been developed by combining domain decomposition with local spectral representations \cite{Zhao2026local}.
By transferring each subinterval to a fixed reference domain, localized constructions allow different regions to accommodate local frequency characteristics while preserving stability and approximation accuracy.
These features suggest a natural route to quadrature schemes that remain compatible with uniform sampling and can better handle oscillatory or nonuniform phase structures.

Motivated by these advances, we develop a numerical integration framework based on local Fourier extension (LFE) approximations.
The proposed strategy constructs, on each subinterval, a stabilized local Fourier continuation of the integrand on an extended periodic domain, and then evaluates the integral \emph{analytically} from the resulting Fourier coefficients.
This yields a direct connection between local spectral reconstruction and numerical quadrature.
In contrast to polynomial-based composite rules on uniform grids, the method is intrinsically adapted to oscillatory and variable-frequency integrands and can reach near machine precision with relatively few sampling points.
An important feature is that the method operates directly on equispaced nodes: under uniform sampling, the discrete LFE matrix (and its truncated-SVD factors) can be precomputed once and reused across all windows, leading to an efficient offline/online implementation that remains compatible with classical composite rules.

In addition to smooth and oscillatory integrands, practical computations often involve \emph{continuous piecewise smooth} functions, where the function is continuous but its derivative contains localized singularities.
Such localized nonsmoothness may strongly perturb local spectral reconstructions and deteriorate quadrature accuracy if it occurs inside a window.
To address this issue, we further develop a singular-point detection and correction strategy within the LFE framework.
The method identifies singularity-containing windows through abnormal growth of a coefficient-energy indicator, localizes the singular point within one grid cell, and corrects only the affected window contribution by a piecewise LFE integration.
This extends the applicability of the proposed quadrature beyond globally smooth integrands and enables near-spectral accuracy to be recovered for continuous piecewise smooth functions, including cases where the singularity is not aligned with window endpoints.

The remainder of this paper is organized as follows.
Section~\ref{SEC2} introduces the proposed numerical integration framework based on multi-interval Fourier extension, including the local approximation construction, the derivation of the quadrature formula, and the associated convergence analysis.
Section~\ref{SEC3} discusses the numerical implementation under uniform sampling, with particular emphasis on matrix precomputation, singular-point detection, and correction for continuous piecewise smooth integrands.
Section~\ref{SEC4} presents numerical experiments for smooth, oscillatory, and piecewise smooth test functions, including a comparison with Clenshaw--Curtis quadrature.
Finally, concluding remarks are given in the last section.

\section{A numerical integration method based on multi-interval Fourier extension}\label{SEC2}

In this section, we present the proposed numerical quadrature framework
based on the local Fourier extension (LFE) approximation.
The construction consists of three components:
(i) local Fourier extension approximation of the integrand,
(ii) analytical evaluation of the integral from the resulting local Fourier
coefficients, and
(iii) convergence analysis of the induced quadrature rule.

\subsection{Local Fourier extension approximation}

We briefly recall the local Fourier extension method introduced in
\cite{Zhao2026local}, which serves as the approximation foundation of the
present quadrature scheme.

Let $I=[a,b]$ and let
\[
a=a_0<a_1<\cdots<a_K=b
\]
be a partition of $I$ into $K$ subintervals
\[
I_k=[a_{k-1},a_k],\qquad k=1,\ldots,K.
\]
For a given function $f$ on $I$, we denote its restriction to $I_k$ by
\[
f_k(x)=f(x),\qquad x\in I_k.
\]

Following the Fourier extension framework, each subinterval is mapped onto
the reference interval
\[
\Lambda=[0,2\pi/T],\qquad T>1,
\]
through the scaling
\[
g_k(t)=f_k(a_{k-1}+s_k t),\qquad
s_k=\frac{T}{2\pi}(a_k-a_{k-1}).
\]
Thus, each local approximation problem is transferred from $I_k$ to the
same fixed interval $\Lambda$, which is essential for the reuse of the
discrete FE matrix in the numerical implementation.

Let
\[
\phi_\ell(t)=e^{i\ell t},\qquad |\ell|\le N,
\]
and define
\[
\Phi_N=\mathrm{span}\{\phi_\ell:\ |\ell|\le N\}.
\]
We introduce the finite-rank synthesis operator
\[
\mathcal{F}_N:\mathbb{C}^{2N+1}\rightarrow \Phi_N,\qquad
{\bf c}_N=\{c_\ell\}_{|\ell|\le N}\mapsto \sum_{\ell=-N}^{N}c_\ell\phi_\ell.
\]
Let $(\sigma_j,v_j,u_j)$ be a singular system of $\mathcal{F}_N$.
For a prescribed tolerance $\epsilon>0$, the truncated singular value
decomposition (TSVD) coefficient vector is defined by
\[
{\bf c}_{N}^{\epsilon,k}
=
\sum_{\sigma_j>\epsilon}
\frac{\langle g_k,u_j\rangle}{\sigma_j}\,v_j.
\]
The corresponding local Fourier extension approximation of $g_k$ is
\[
Q_N^\epsilon g_k(t)
=
\mathcal{F}_N{\bf c}_{N}^{\epsilon,k}
=
\sum_{\ell=-N}^{N}c_{\ell}^{\epsilon,k}\phi_\ell(t).
\]

The global LFE approximation of $f$ on $I$ is then defined piecewise by
\[
(P_{N,K}^{\epsilon}f)(x)
=
(Q_N^\epsilon g_k)\!\left(\frac{x-a_{k-1}}{s_k}\right),
\qquad x\in I_k,\quad k=1,\ldots,K.
\]

The following estimate, inherited from \cite{Zhao2026local}, describes the
local approximation error.

\begin{lem}\cite{Zhao2026local}\label{lemma1}
For each $k=1,\ldots,K$, the local approximation error satisfies
\begin{equation}\label{localapproerror}
\|g_k-Q_N^\epsilon g_k\|_{L^2(\Lambda)}
\le
\inf\Big\{
\|g_k-\mathcal{F}_N{\bf c}_N\|_{L^2(\Lambda)}
+\epsilon\|{\bf c}_N\|_{\ell^2}
:\ {\bf c}_N\in\mathbb{C}^{2N+1}
\Big\}.
\end{equation}
\end{lem}

Lemma~\ref{lemma1} shows that the LFE approximation inherits the
best-approximation property up to the regularization term
$\epsilon\|{\bf c}_N\|_{\ell^2}$.
This estimate will be the key tool in the subsequent quadrature error
analysis.

\subsection{Derivation of the quadrature formula}

The central idea of the proposed method is to integrate the local Fourier
extension approximation analytically.

The target integral is
\[
\mathcal{I}(f)=\int_a^b f(x)\,dx.
\]
We define the LFE quadrature approximation by
\[
\mathcal{Q}_{N,K}^{\epsilon}(f)
=
\int_a^b (P_{N,K}^{\epsilon}f)(x)\,dx.
\]
Using the piecewise definition of $P_{N,K}^{\epsilon}f$, we obtain
\[
\mathcal{Q}_{N,K}^{\epsilon}(f)
=
\sum_{k=1}^{K}
\int_{I_k}
(Q_N^\epsilon g_k)\!\left(\frac{x-a_{k-1}}{s_k}\right)\,dx.
\]
Applying the change of variables
\[
x=a_{k-1}+s_k t,\qquad dx=s_k\,dt,
\]
yields
\[
\mathcal{Q}_{N,K}^{\epsilon}(f)
=
\sum_{k=1}^{K}
s_k
\int_{\Lambda}
Q_N^\epsilon g_k(t)\,dt.
\]

Substituting the Fourier expansion
\[
Q_N^\epsilon g_k(t)=\sum_{|\ell|\le N}c_\ell^{\epsilon,k}\phi_\ell(t)
\]
gives
\[
\mathcal{Q}_{N,K}^{\epsilon}(f)
=
\sum_{k=1}^{K}
s_k
\sum_{|\ell|\le N}
c_\ell^{\epsilon,k}\,\omega_\ell,
\]
where
\[
\omega_\ell
=
\int_0^{2\pi/T}e^{i\ell t}\,dt
=
\begin{cases}
\dfrac{2\pi}{T}, & \ell=0,\\[6pt]
\dfrac{e^{i\ell 2\pi/T}-1}{i\ell}, & \ell\neq 0.
\end{cases}
\]

Therefore, once the local coefficient vectors
${\bf c}_{N}^{\epsilon,k}$ are computed, the quadrature value is obtained by
a direct weighted summation of the Fourier coefficients.
This analytic evaluation is the fundamental reason why the method can
achieve very high accuracy on uniform samples.

\subsection{Convergence analysis}

We now establish the convergence properties of the proposed quadrature rule.

\begin{thm}[Quadrature error is bounded by the LFE approximation error]
\label{thm:main_bound}
For any $f\in L^2(a,b)$,
\[
\left|
\int_a^b f(x)\,dx-\mathcal{Q}_{N,K}^{\epsilon}(f)
\right|
\le
\sqrt{b-a}\,
\|f-P_{N,K}^{\epsilon}f\|_{L^2(a,b)}.
\]
\end{thm}

\begin{pf}
By definition of $\mathcal{Q}_{N,K}^{\epsilon}(f)$,
\[
\int_a^b f(x)\,dx-\mathcal{Q}_{N,K}^{\epsilon}(f)
=
\int_a^b \bigl(f(x)-P_{N,K}^{\epsilon}f(x)\bigr)\,dx.
\]
The result follows immediately from the Cauchy--Schwarz inequality.
\end{pf}

Theorem~\ref{thm:main_bound} reduces the quadrature error to the
approximation error of the LFE reconstruction.
Hence, any convergence estimate for $P_{N,K}^{\epsilon}f$
directly yields a corresponding estimate for
$\mathcal{Q}_{N,K}^{\epsilon}(f)$.

\begin{thm}[Algebraic convergence]
\label{thm:Hp}
Assume that $f\in H^p(a,b)$ with $p\ge 1$.
Then there exists a constant $C>0$, independent of $N$ and $K$, such that
\[
\left|
\int_a^b f(x)\,dx-\mathcal{Q}_{N,K}^{\epsilon}(f)
\right|
\le
C(b-a)\,
N^{-p}
\max_{1\le k\le K}
\left(
s_k^{\,p-\frac12}\|f^{(p)}\|_{L^2(I_k)}
\right)
+\mathcal{O}(\epsilon).
\]
\end{thm}

\begin{pf}
By Theorem~\ref{thm:main_bound}, it suffices to estimate
$\|f-P_{N,K}^{\epsilon}f\|_{L^2(a,b)}$.
Since $P_{N,K}^{\epsilon}f$ is defined piecewise, we have
\[
\|f-P_{N,K}^{\epsilon}f\|_{L^2(a,b)}^2
=
\sum_{k=1}^{K}
\|f-P_{N,K}^{\epsilon}f\|_{L^2(I_k)}^2.
\]
For each $k$,
\[
\|f-P_{N,K}^{\epsilon}f\|_{L^2(I_k)}
=
s_k^{1/2}\|g_k-Q_N^\epsilon g_k\|_{L^2(\Lambda)}.
\]
Applying Lemma~\ref{lemma1} together with the standard trigonometric
approximation estimate on $\Lambda$ yields
\[
\|g_k-Q_N^\epsilon g_k\|_{L^2(\Lambda)}
\le
C N^{-p}\|g_k^{(p)}\|_{L^2(\Lambda)}+\mathcal{O}(\epsilon).
\]
Moreover, by the chain rule,
\[
g_k^{(p)}(t)
=
s_k^p\,f^{(p)}(a_{k-1}+s_k t),
\]
and hence
\[
\|g_k^{(p)}\|_{L^2(\Lambda)}
=
s_k^{\,p-\frac12}\|f^{(p)}\|_{L^2(I_k)}.
\]
Combining the above estimates and using
\[
\|f-P_{N,K}^{\epsilon}f\|_{L^2(a,b)}
\le
\sqrt{b-a}\,
\max_{1\le k\le K}
\|f-P_{N,K}^{\epsilon}f\|_{L^2(I_k)},
\]
we obtain the stated result.
\end{pf}

The above theorem shows that the proposed quadrature inherits the local
approximation properties of the LFE reconstruction.
In particular, the convergence improves as the local windows become
shorter, which is consistent with the main motivation of the multi-interval
construction.

\begin{corollary}[Uniform partition rate]
\label{cor:uniform_partition}
If the partition is uniform, i.e.,
\[
a_k-a_{k-1}=\frac{b-a}{K},\qquad k=1,\ldots,K,
\]
then
\[
\left|
\int_a^b f(x)\,dx-\mathcal{Q}_{N,K}^{\epsilon}(f)
\right|
\le
C(T,p)\,
(b-a)^{p+\frac12}\,
N^{-p}K^{-p}\,
\|f^{(p)}\|_{L^2(a,b)}.
\]
\end{corollary}

Corollary~\ref{cor:uniform_partition} makes explicit the gain achieved by
the local partition: for fixed $N$, reducing the window size through larger
$K$ improves the quadrature accuracy algebraically.
This observation also explains the effectiveness of the proposed method for
highly oscillatory and variable-frequency integrands, which will be
confirmed numerically in Section~\ref{SEC4}.

\section{Numerical implementation based on uniform sampling}\label{SEC3}

In this section we describe the practical implementation of the proposed
multi-interval Fourier extension quadrature under uniform sampling.
The implementation consists of an offline precomputation stage and an
online evaluation stage.
Because the same equispaced sampling pattern is used on every local window,
the core discrete Fourier extension matrix and its TSVD factors can be
precomputed once and then reused throughout the computation.
This is the key mechanism that makes the method efficient on uniformly
sampled data.

Throughout this work, for convenience in the window-based implementation and
the subsequent singularity correction procedure, we fix the Fourier extension
parameters as
\[
n=10,\qquad m=2n+1=21,\qquad T=6,\qquad \epsilon=10^{-15},
\]
where $\epsilon$ is chosen close to machine precision.
This parameter set provides a stable local reconstruction while keeping the
window size small enough for efficient sliding-window computations and local
singularity treatment.

\subsection{Matrix precomputation on each subinterval}

Let each subinterval $I_k=[a_{k-1},a_k]$ be sampled by $m$ equidistant nodes
\[
x_{k,i}=a_{k-1}+ih_k,\qquad
h_k=\frac{a_k-a_{k-1}}{m-1},\qquad i=0,1,\ldots,m-1.
\]
Under the scaling
\[
t_i=\frac{x_{k,i}-a_{k-1}}{s_k},\qquad
s_k=\frac{T}{2\pi}(a_k-a_{k-1}),
\]
the local discrete Fourier extension system can be written as
\[
\mathbf{F}_{n}\mathbf{c}_k\approx \mathbf{g}_k,
\]
where
\[
(\mathbf{F}_{n})_{i,\ell}
=
\frac{1}{\sqrt{L}}e^{i\ell t_i},
\qquad
\mathbf{g}_k=
\bigl(g_k(t_0),g_k(t_1),\ldots,g_k(t_{m-1})\bigr)^\top,
\qquad |\ell|\le n.
\]
Here $L=T(m-1)$ is chosen so that the reference sampling interval exactly
matches $[0,2\pi/T]$.

Since the nodes $\{t_i\}$ are equispaced and the parameters $(n,m,T)$ are
fixed, the matrix $\mathbf{F}_{n}$ is identical for all local windows.
Therefore its truncated singular value decomposition
\[
\mathbf{F}_{n}=\mathbf{U}\mathbf{\Sigma}\mathbf{V}^\ast
\]
can be computed once offline and reused everywhere.
Let
\[
\mathbf{\Sigma}_\epsilon^\dagger
=
\mathrm{diag}\bigl(\sigma_j^{-1}\mathbbm{1}_{\{\sigma_j>\epsilon\}}\bigr)
\]
denote the truncated pseudoinverse.
Then the TSVD coefficient vector is given by
\begin{equation}\label{eq:tsvd_coeff}
\mathbf{c}_k^\epsilon
=
\mathbf{V}\mathbf{\Sigma}_\epsilon^\dagger\mathbf{U}^\ast\mathbf{g}_k.
\end{equation}

\paragraph{Important implementation note.}
Although \eqref{eq:tsvd_coeff} is linear in $\mathbf{g}_k$, the computation
must be carried out in the TSVD order
\[
\mathbf{y}_k=\mathbf{U}^\ast\mathbf{g}_k,\qquad
\mathbf{z}_k=\mathbf{\Sigma}_\epsilon^\dagger\mathbf{y}_k,\qquad
\mathbf{c}_k^\epsilon=\mathbf{V}\mathbf{z}_k,
\]
rather than by explicitly forming the merged matrix
$\mathbf{V}\mathbf{\Sigma}_\epsilon^\dagger\mathbf{U}^\ast$.
This ordering is numerically more stable: the projection
$\mathbf{U}^\ast\mathbf{g}_k$ first produces a decaying modal representation,
and only afterwards are the retained singular directions rescaled.
In floating-point arithmetic, directly fusing the three factors may lead to
unnecessary amplification of roundoff errors.

\paragraph{Quadrature assembly.}
Let
\[
\boldsymbol{\omega}=(\omega_{-n},\ldots,\omega_n)^\top\in\mathbb{C}^{2n+1}
\]
be the vector of exact Fourier-mode integrals on the reference interval,
\[
\omega_\ell=\int_0^{2\pi/T}e^{i\ell t}\,dt.
\]
Then the contribution of the $k$-th local window is
\[
\mathcal{Q}_k^\epsilon(f)
=
s_k\,\boldsymbol{\omega}^\ast\mathbf{c}_k^\epsilon
=
s_k\,\boldsymbol{\omega}^\ast
\mathbf{V}\mathbf{\Sigma}_\epsilon^\dagger\mathbf{U}^\ast\mathbf{g}_k,
\]
and the global quadrature is obtained by summing all local contributions:
\[
\mathcal{Q}_{n,K}^\epsilon(f)=\sum_{k=1}^K \mathcal{Q}_k^\epsilon(f).
\]

\paragraph{Offline/online separation.}
The matrices $(\mathbf{U},\mathbf{\Sigma},\mathbf{V})$ depend only on the
fixed parameters $(n,m,T)$ and the reference equispaced sampling pattern.
Hence they can be precomputed once offline.
Online evaluation then consists of:
(i) assembling the local data vector $\mathbf{g}_k$,
(ii) multiplying by $\mathbf{U}^\ast$,
(iii) diagonal scaling by $\mathbf{\Sigma}_\epsilon^\dagger$,
(iv) multiplying by $\mathbf{V}$, and
(v) taking one inner product with $\boldsymbol{\omega}$.
This structure makes the implementation especially efficient when a large
number of windows must be processed.

\subsection{Uniform sampling strategy}

We now describe the global implementation on a uniform grid.
Let
\[
x_j=a+jh,\qquad j=0,1,\ldots,M,
\]
be the equidistant nodes on $[a,b]$, where
\[
h=\frac{b-a}{M}.
\]
Since $m=21$, each local window contains $21$ consecutive grid points,
and adjacent windows overlap by one endpoint.
Equivalently, the window shift is
\[
m-1=20.
\]

Depending on the total number of sampling nodes $M+1$, three cases arise.

\paragraph{Case 1: $M+1<21$.}
If fewer than $21$ nodes are available, the standard precomputed matrix
cannot be used directly.
In this case we form a smaller discrete Fourier extension matrix using all
available nodes and compute its TSVD once for that particular interval.
This corresponds to a single global local-FE reconstruction.

\paragraph{Case 2: $M+1=21$.}
If the total number of nodes equals $21$, the global interval itself forms
exactly one local window.
The precomputed reference matrix can then be reused directly, and the
integral is obtained from one local Fourier extension reconstruction.

\paragraph{Case 3: $M+1>21$.}
If more than $21$ nodes are available, we employ a sliding-window strategy
with window size $21$:
\[
\{x_0,\ldots,x_{20}\},\quad
\{x_{20},\ldots,x_{40}\},\quad
\{x_{40},\ldots,x_{60}\},\ \ldots
\]
Thus, each window contributes the integral over one block of length
$(m-1)h=20h$, while neighboring windows share the boundary node.

If the last block contains fewer than $21$ nodes, we borrow the necessary
number of points from the preceding part of the grid so that the final
window still contains exactly $21$ equispaced samples.
The window length therefore remains unchanged, so the same precomputed SVD
can still be reused.
The only difference is that, for the last window, the effective integration
range in the reference variable is no longer the full interval
$[0,2\pi/T]$, but a truncated subinterval
\[
[t_0,\,2\pi/T]\subset[0,2\pi/T],
\]
where $t_0$ is determined by the number of remaining subintervals.
Accordingly, the Fourier-mode integrals are replaced by
\[
\omega_\ell^{(\mathrm{tail})}
=
\int_{t_0}^{2\pi/T}e^{i\ell t}\,dt.
\]

\paragraph{Algorithmic structure.}
Under global uniform sampling with $M+1>21$, the quadrature is computed as
follows:
\begin{enumerate}
\item Partition the global grid into overlapping windows of size $21$;
\item For each window, assemble the local data vector $\mathbf{g}_k$;
\item Compute $\mathbf{c}_k^\epsilon$ via the stable TSVD pipeline
      $\mathbf{U}^\ast\rightarrow \mathbf{\Sigma}_\epsilon^\dagger\rightarrow \mathbf{V}$;
\item Evaluate the corresponding local integral analytically from the
      Fourier coefficients;
\item Sum all local contributions.
\end{enumerate}

This implementation preserves numerical stability while allowing all
SVD-related quantities to be precomputed once and reused globally.
Moreover, since the method works directly on equispaced nodes, it can be
fairly compared with standard composite rules on uniform grids.

\subsection{Singular point detection, subgrid localization and correction for continuous piecewise smooth integrands}
\label{subsec:edge_refined}

We now describe the correction procedure for continuous piecewise smooth
integrands whose singularities are associated with derivative
discontinuities.
The main difficulty is that, if a singular point lies inside a local window,
then the corresponding local Fourier extension problem is no longer fitted
to a single smooth branch, and the quadrature contribution of that window
may lose its high accuracy.
To overcome this difficulty, we detect the affected window, localize the
singular point, reconstruct one-sided smooth models, and replace only the
contaminated local contribution.

\paragraph{Step 1: detection of singularity-containing windows.}
Let
\[
I_k=[x_{k,0},x_{k,m-1}]
\]
denote the $k$-th sliding window and let ${\bf c}_k^\epsilon$ be the
associated TSVD coefficient vector.
We define the coefficient-energy indicator
\[
\eta_k=\|{\bf c}_k^\epsilon\|_2.
\]
On smooth windows $\eta_k$ remains moderate, while the presence of a
singular point inside $I_k$ causes strong amplification because the local
FE problem becomes more severely ill-conditioned.
Hence, windows with abnormally large $\eta_k$ are identified as candidates
containing singular points.

\paragraph{Step 2: localization of the singular cell.}
Assume that $I_{k_0}$ is detected.
To localize the singular point within this window, we consider all interior
splitting positions $i=1,\ldots,m-2$.
For each $i$, we form two overlapping candidate windows and compute the
corresponding coefficient vectors ${\bf cL}_i$ and ${\bf cR}_i$.
We then select
\begin{equation}
i_0
=
\arg\min_{1\le i\le m-2}
\Bigl(
\|{\bf cL}_i\|_2+\|{\bf cR}_i\|_2
\Bigr).
\label{eq:min_split}
\end{equation}
This determines the smallest grid cell containing the singular point.
Equivalently, if the corresponding global index is denoted by
$\texttt{location}$, then the singular point lies in
\[
[x_{\texttt{location}-1},\,x_{\texttt{location}}].
\]

\paragraph{Step 3: prediction of one-sided endpoint values.}
If the singular point does not coincide with a grid node, then each of the
two local windows used for the left and right reconstructions contains one
contaminated endpoint.
To recover one-sided smooth data, we use the smallest singular mode of the
discrete FE matrix
\[
{\bf A}={\bf U}\Sigma{\bf V}^\ast.
\]
Let ${\bf u}_{\min}$ be the left singular vector associated with the
smallest singular value $\sigma_{\min}$.
For a smooth window,
\[
{\bf f}={\bf A}{\bf c}+{\bf r},
\]
and hence
\[
{\bf u}_{\min}^\ast{\bf f}
=
\sigma_{\min}{\bf v}_{\min}^\ast{\bf c}
+
{\bf u}_{\min}^\ast{\bf r},
\qquad
|{\bf u}_{\min}^\ast{\bf f}|
\le
\sigma_{\min}\|{\bf c}\|_2+\|{\bf r}\|_2.
\]
Thus, for one-sided smooth data, the component along the smallest singular
direction should be negligible, which motivates the constraint
\[
{\bf u}_{\min}^\ast{\bf f}\approx 0.
\]

If one endpoint value is unknown, we write
\[
{\bf f}(\alpha)=\widetilde{\bf f}+\alpha{\bf e}_p,
\]
where $\widetilde{\bf f}$ contains the $m-1$ reliable samples.
Imposing
\[
{\bf u}_{\min}^\ast{\bf f}(\alpha)=0
\]
yields the linear predictor
\begin{equation}
\alpha
=
-\frac{{\bf u}_{\min}^\ast\widetilde{\bf f}}
       {{\bf u}_{\min}^\ast{\bf e}_p}.
\label{eq:endpoint_predictor}
\end{equation}
In practice, we use \eqref{eq:endpoint_predictor} to predict the left limit
at $x_{\texttt{location}}$ and the right limit at
$x_{\texttt{location}-1}$.
Numerically, the prediction is highly accurate and is sufficient for the
subsequent quadrature correction.

\paragraph{Step 4: subgrid singular point estimation.}
After replacing the contaminated endpoint values by their predicted
one-sided limits, we obtain two smooth local Fourier extension models,
denoted by $f_L(x)$ and $f_R(x)$, corresponding to the left and right
smooth branches.
The singular point $\xi$ is then estimated as the unique root of
\[
f_L(x)-f_R(x)=0
\]
inside the cell
\[
[x_{\texttt{location}-1},\,x_{\texttt{location}}].
\]
Since $f_L$ and $f_R$ are explicit local Fourier expansions, this root can
be computed efficiently by bisection or a safeguarded Newton iteration.
The resulting estimate $\widehat{\xi}$ provides subgrid accuracy.

\paragraph{Step 5: local integral replacement on the singular window.}
Let $I_k=[a_k,b_k]$ denote the original window identified as singular.
Instead of using its contaminated quadrature contribution, we replace it by
\begin{equation}
\int_{a_k}^{b_k}f(x)\,dx
=
\int_{a_k}^{\widehat{\xi}}f_L(x)\,dx
+
\int_{\widehat{\xi}}^{b_k}f_R(x)\,dx.
\label{eq:local_replacement}
\end{equation}
Each subintegral is evaluated analytically from the corresponding local
Fourier coefficients by using the continuous primitives of the Fourier
modes.
Since only the affected window is corrected, the global structure of the
uniform-sampling implementation is preserved, while the loss of accuracy
caused by the singular point is removed.

\subsection*{Algorithm: quadrature with singular point correction}

\begin{enumerate}
\item Compute the coefficient energies $\eta_k$ for all windows and detect
      singularity-containing candidates.
\item For each detected window, localize the grid cell containing the
      singular point via \eqref{eq:min_split}.
\item Predict the missing one-sided endpoint values from the smallest
      singular mode by \eqref{eq:endpoint_predictor}.
\item Construct the left and right local FE models and estimate the subgrid
      singular position $\widehat{\xi}$.
\item Replace the original contribution of the detected window by the
      piecewise integral \eqref{eq:local_replacement}.
\item Sum all corrected and uncorrected window contributions.
\end{enumerate}

After this correction, the effective local contributions are again computed
from smooth one-sided data, and the quadrature regains the high-order
accuracy predicted by the smooth-case theory.

\begin{rem}[Scope]
The above correction strategy is designed for continuous piecewise smooth
functions, in particular for singularities of low-regularity type such as
discontinuities in derivatives.
Jump discontinuities are not considered here, since their one-sided limits
cannot be recovered from pointwise samples alone without additional
information.
\end{rem}
\section{Numerical Experiments}\label{SEC4}

This section presents numerical experiments to assess the accuracy of the
proposed quadrature method based on local Fourier extension (LFE).
All tests are carried out under a unified discretization framework so that
the comparison with classical composite quadrature rules is performed on the
same set of equispaced nodes.

\subsection{Experimental setup}

\paragraph{Parameter configuration.}
Following the implementation described in Section~\ref{SEC3}, we fix
\[
n=10,\qquad m=2n+1=21,\qquad T=6,\qquad \epsilon=10^{-15}.
\]
Accordingly, each local window contains $21$ equidistant nodes, and adjacent
windows overlap by one endpoint, so that the window shift is $m-1=20$.
The corresponding discrete Fourier extension matrix and its truncated singular
value decomposition are precomputed once and reused throughout all experiments.

\paragraph{Global discretization.}
Let $[a,b]$ be the integration interval.
The global discretization consists of $M+1$ equidistant nodes
\[
x_j=a+jh,\qquad j=0,1,\ldots,M,\qquad h=\frac{b-a}{M}.
\]
Thus, the total number of function evaluations is $M+1$.
When $M$ is not divisible by $m-1=20$, the last local contribution is
computed by the tail-window strategy described in Section~\ref{SEC3}.

\paragraph{Comparison method.}
To evaluate the performance of the proposed method on uniform grids, we
compare it with the composite Simpson rule, which is a standard fourth-order
method for smooth integrands on equispaced nodes.
All methods use exactly the same sampling points.
To preserve the nominal fourth-order accuracy of Simpson's rule, we choose
$M$ to be even in all tests.

\paragraph{Error measurement.}
Since all test functions considered below admit closed-form integrals, the
reference value $I_{\mathrm{exact}}$ is known explicitly.
For all examples, we report the absolute error
\[
E=|Q-I_{\mathrm{exact}}|.
\]

\paragraph{Computational environment.}
All computations are performed in MATLAB 2016b on a Windows 10 platform
with 16\,GB memory and an Intel(R) Core(TM) i7 CPU.

\paragraph{Performance metric.}
The principal performance metric used in this section is the error decay with
respect to the total number of uniform subintervals $M$, equivalently the
total number of nodes $M+1$.

\subsection{Smooth test functions}

We first examine the performance of the proposed quadrature method for
globally smooth integrands. In this regime, the composite Simpson rule
attains its standard fourth-order convergence and therefore provides a
natural baseline for comparison.

\paragraph{Test functions.}
We consider the following analytic functions on non-symmetric intervals:
\begin{align*}
f_1(x)&=3x^2-e^{-x}-2\sin(2x),&x\in[0.1,1.5],\\
f_2(x)&=e^{x}\cos(3x)+\frac{x^2}{1+x},&x\in[0.2,1.3],\\
f_3(x)&=\frac{1}{1+x^2}+2\cos(\sin(2x))\cos(2x),&x\in[-0.1,1.4].
\end{align*}
These functions are analytic on their respective intervals and do not possess
special symmetry or periodicity that would favor any particular quadrature
formula.
The exact integral values are obtained from their closed-form antiderivatives.

\paragraph{Results and discussion.}
Figure~\ref{fig1} reports the error decay versus the number of uniform
subintervals $M$ for the three smooth test functions.
As expected, the composite Simpson rule exhibits fourth-order convergence
before entering the roundoff-dominated regime.
\begin{figure}[htbp]
	\begin{center}
		\subfigure[$f_1(x)$]{\resizebox*{6cm}{3cm}{\label{subfg11}\includegraphics{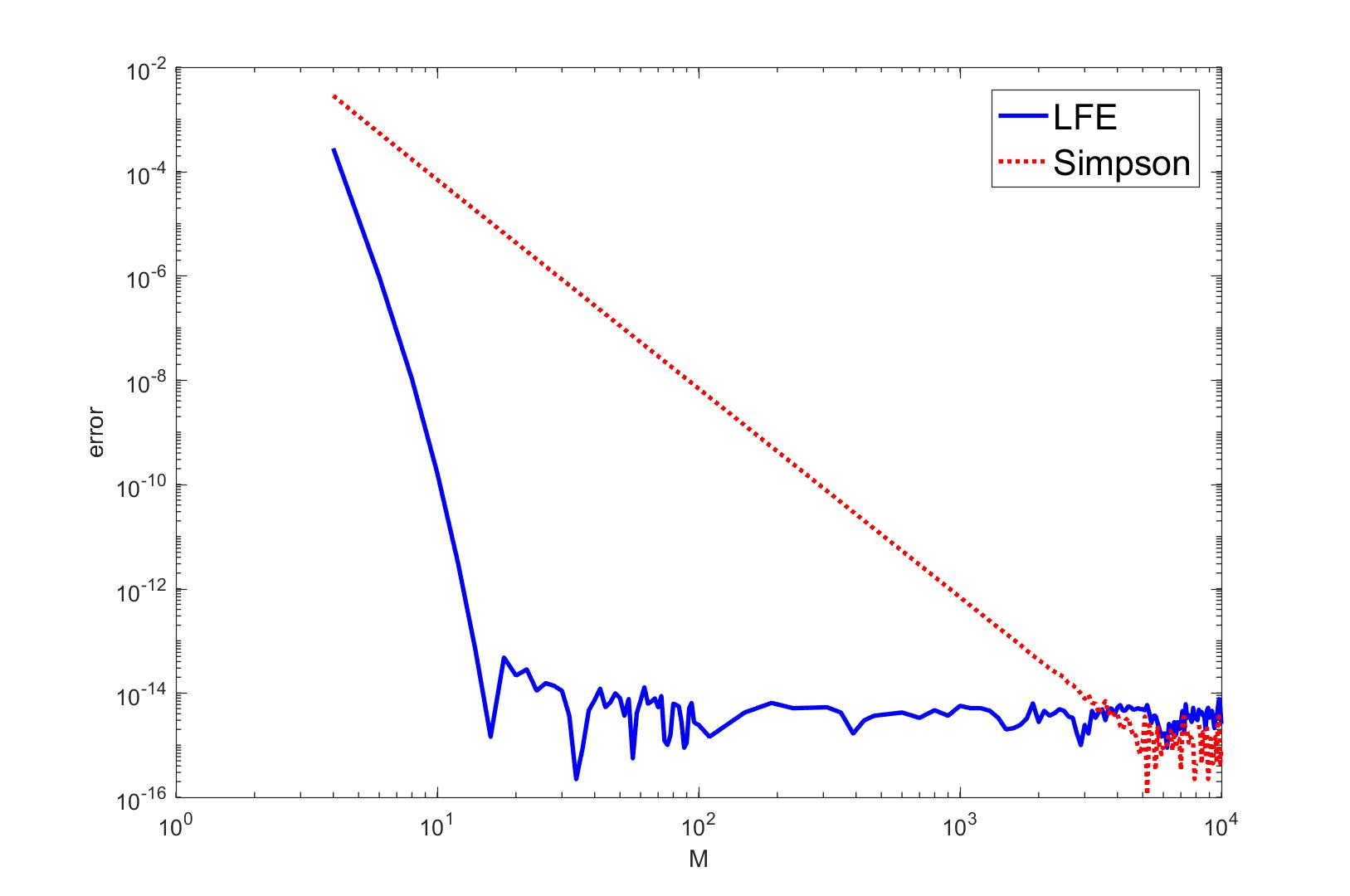}}}
		\subfigure[$f_2(x)$]{\resizebox*{6cm}{3cm}{\label{subfg12}\includegraphics{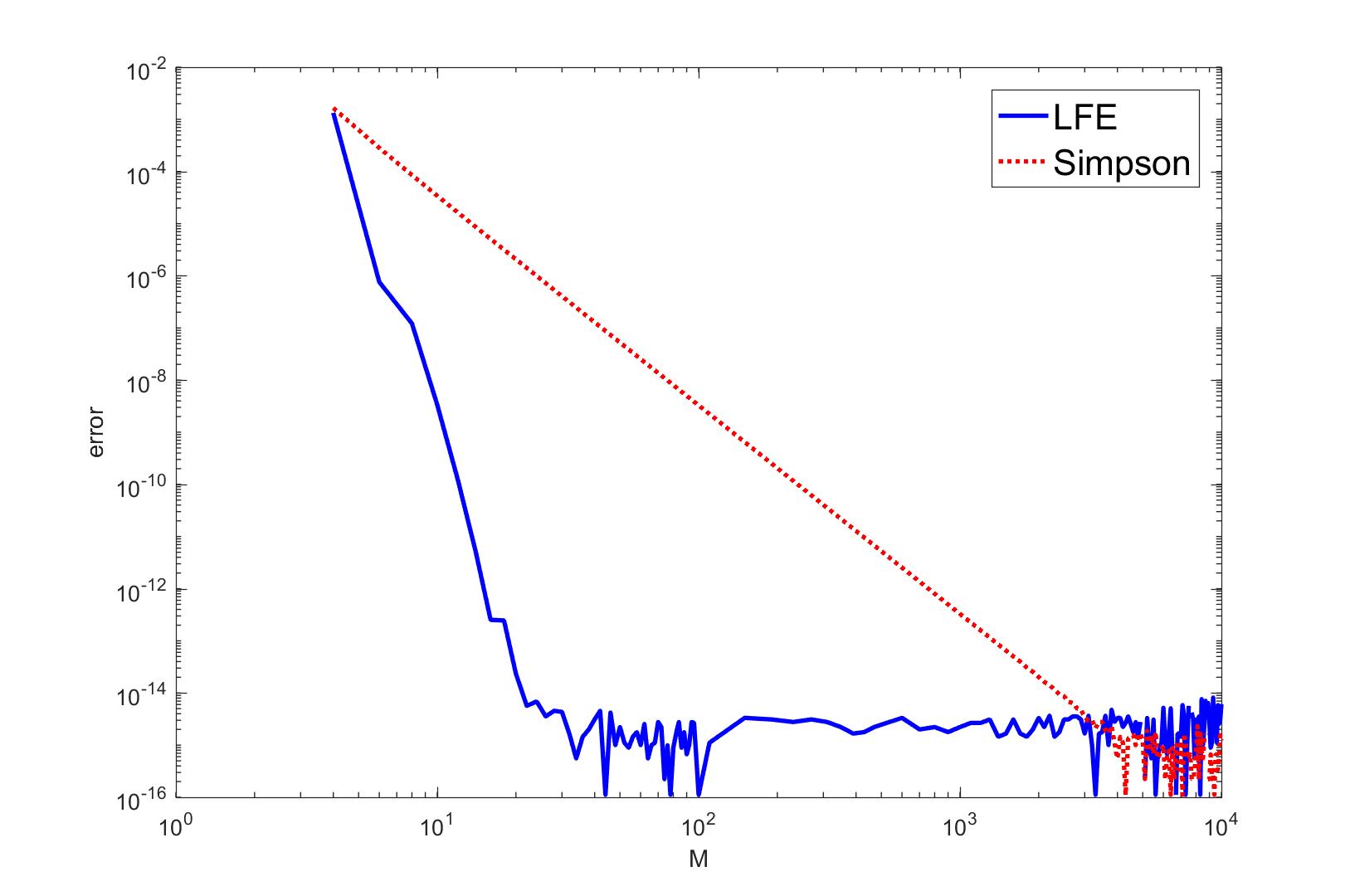}}}
        \subfigure[$f_3(x)$]{\resizebox*{6cm}{3cm}{\label{subfg13}\includegraphics{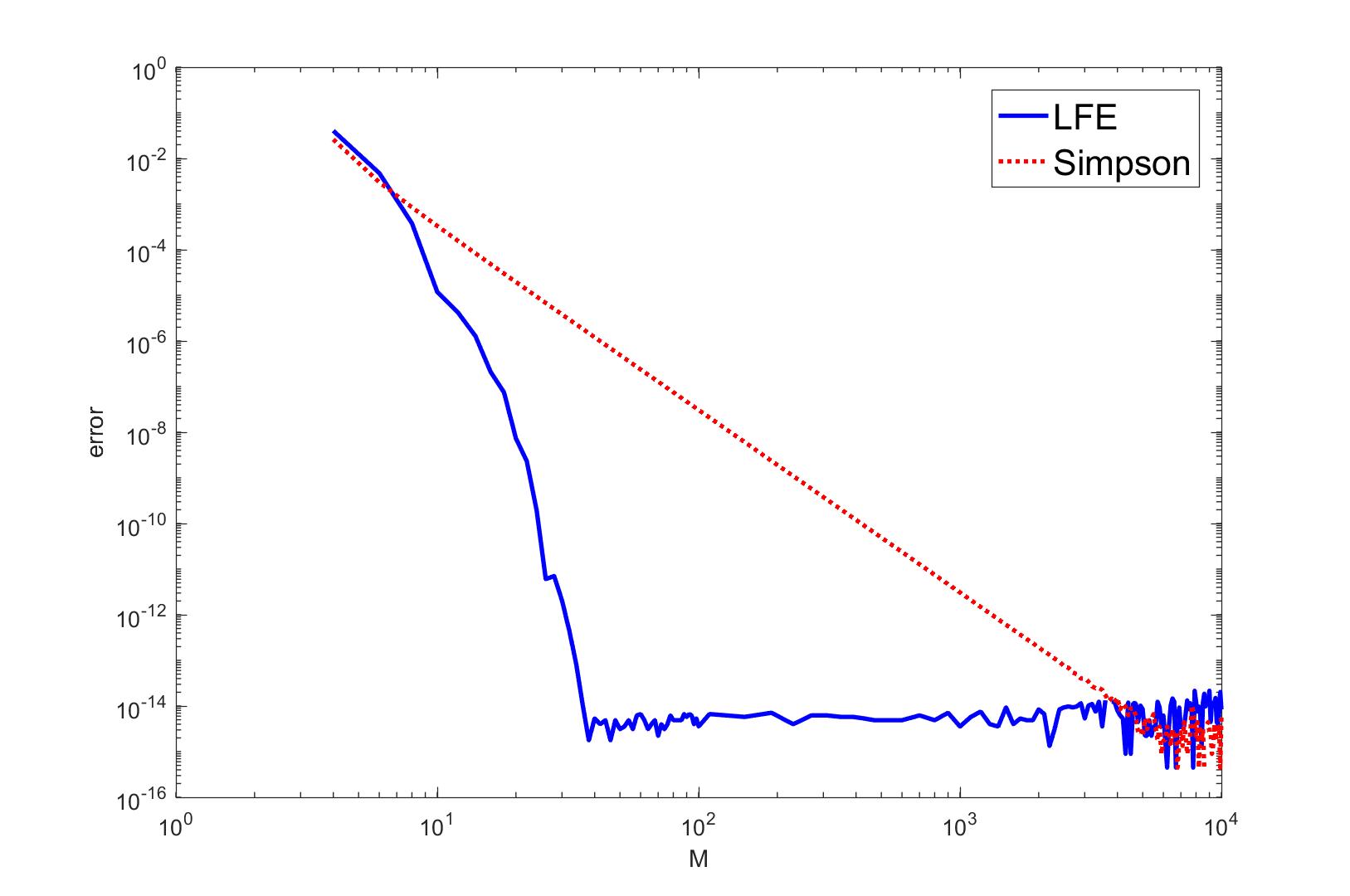}}}
		\caption{Error decay versus the total number of uniform subintervals $M$
        for the smooth test functions $f_1$--$f_3$.}
        \label{fig1}
	\end{center}
\end{figure}
In contrast, the proposed LFE quadrature reduces the error much more rapidly
for all three tests and reaches a near-machine-precision plateau with
substantially fewer nodes.
This advantage is quantified in Table~\ref{tab1}.
For example, to achieve $E\le 10^{-12}$, the proposed method requires only
$M=14,16,32$ for $f_1,f_2,f_3$, respectively, while the composite Simpson
rule requires $M=912,758,1324$.
Similar savings are observed for the target levels $10^{-8}$ and $10^{-10}$.

These experiments show that, even on globally smooth functions, the proposed
LFE-based quadrature provides a substantial gain in accuracy on uniform grids.

\begin{table}[htbp]
	\begin{center}
		\caption{Minimal number of uniform subintervals $M$ required to achieve a prescribed
        error level for the smooth test functions $f_1$--$f_3$.}
        \label{tab1}
		\small\tabcolsep 0.4pt
		{\begin{tabular*}{\textwidth}{@{\extracolsep\fill}ccccccccccccc}
				\toprule
				\multirow{2}{*}{$E$}
				&\multicolumn{2}{c}{$f_1(x)$}
				&\multicolumn{2}{c}{$f_2(x)$}
                &\multicolumn{2}{c}{$f_3(x)$}\\
				\cline{2-3} \cline{4-5}\cline{6-7}
				&  LFE&Simpson&  LFE&Simpson&  LFE&Simpson\\
				\midrule
				1e-8&10&92&10&76&20&134\\
                1e-10&12&288&14&240&26&420\\
                1e-12&14&912&16&758&32&1324\\
				\bottomrule
		\end{tabular*}}
	\end{center}
\end{table}

\subsection{Oscillatory and variable-frequency test functions}

We next examine oscillatory and variable-frequency integrands.
All implementation details follow the experimental setting described above.

\paragraph{Test functions.}
We consider the following three representative examples:
\begin{align*}
f_4(x)&=e^{-x}\sin(\omega x), &&x\in[0,1.1],\\
f_5(x)&=-2\kappa x\sin(\kappa x^2), &&x\in[0.2,1.3],\\
f_6(x)&=\frac{2x}{(1+\alpha-x^2)^2}, &&x\in[0,1].
\end{align*}
Their exact primitives are
\[
\int f_4(x)\,dx
=
\frac{e^{-x}\bigl(-\omega\cos(\omega x)-\sin(\omega x)\bigr)}
{1+\omega^2},
\qquad
\int f_5(x)\,dx=\cos(\kappa x^2),
\qquad
\int f_6(x)\,dx=\frac{1}{1+\alpha-x^2}.
\]
\begin{table}[htbp]
	\begin{center}
		\caption{Minimal number of uniform subintervals $M$ required to achieve prescribed
        error levels for the oscillatory and variable-frequency test functions
        $f_4$--$f_6$.}
        \label{tab:osc}
		\small\tabcolsep 0.4pt
		{\begin{tabular*}{\textwidth}{@{\extracolsep\fill}ccccccccccccc}
				\toprule
				\multirow{2}{*}{$E$}
				&\multicolumn{2}{c}{$f_4(x), \omega=100$}
				&\multicolumn{2}{c}{$f_4(x), \omega=200$}
                &\multicolumn{2}{c}{$f_5(x), \kappa=50$}
                &\multicolumn{2}{c}{$f_5(x), \kappa=100$}
				&\multicolumn{2}{c}{$f_6(x), \alpha=0.2$}
                &\multicolumn{2}{c}{$f_6(x), \alpha=0.1$}\\
				\cline{2-3} \cline{4-5}\cline{6-7}\cline{8-9}\cline{10-11}\cline{12-13}
				&  LFE&Simpson&  LFE&Simpson&  LFE&Simpson &  LFE&Simpson&  LFE&Simpson&  LFE&Simpson\\
				\midrule
				1e-8&154&1024&276&1448&228&3844&418&7364&100&948&228&2196\\
                1e-10&178&3232&296&3568&260&12148&478&23268&164&2964&340&6932\\
                1e-12&196&10204&392&14444&308&38340&592&73396&260&9364&500&21876\\
				\bottomrule
		\end{tabular*}}
	\end{center}
\end{table}

\paragraph{Uniform oscillation: $f_4$.}
Figures~\ref{fig:osc_f4}(a)--(b) display the error decay for $\omega=100$
and $\omega=200$.
The composite Simpson rule exhibits the expected fourth-order behavior, but
requires a large number of subintervals before reaching high accuracy.
The LFE quadrature achieves substantially smaller errors with many fewer
nodes.
For instance, to reach $E\le10^{-12}$, the LFE method requires
$M=196$ for $\omega=100$ and $M=392$ for $\omega=200$,
whereas Simpson's rule requires $M=10204$ and $M=14444$,
respectively; see Table~\ref{tab:osc}.

\paragraph{Quadratic-phase oscillation: $f_5$.}
The improvement becomes more pronounced for $f_5$, whose phase is quadratic.
Although the local oscillation rate grows with $x$, the local rescaling on
each window effectively suppresses the higher-order phase variation and makes
the oscillation more regular in local coordinates.
As a result, the global resolution requirement is greatly reduced.
To achieve $E\le10^{-12}$, the LFE method requires $M=308$ for $\kappa=50$
and $M=592$ for $\kappa=100$, whereas Simpson's rule requires
$M=38340$ and $M=73396$, respectively; see Table~\ref{tab:osc}.

\paragraph{Near-singular smooth case: $f_6$.}
The function $f_6$ remains smooth on $[0,1]$, but becomes increasingly
steep as $\alpha$ decreases.
Such behavior typically necessitates substantial mesh refinement for
polynomial-based composite rules.
The local Fourier representation captures this rapid variation much more
efficiently.
For example, when $\alpha=0.1$, the target accuracy $E\le10^{-12}$ is
achieved with $M=500$ for LFE, while Simpson's rule requires $M=21876$;
see Table~\ref{tab:osc}.

\paragraph{Discussion.}
These examples indicate that the proposed LFE quadrature is particularly
effective for oscillatory, variable-frequency, and rapidly varying smooth
integrands.
While it already outperforms the composite Simpson rule for uniformly
oscillatory functions, its advantage becomes significantly stronger when the
local frequency or local variation changes across the interval.

\begin{figure}[htbp]
	\begin{center}
		\subfigure[$\omega=100$]{\resizebox*{6cm}{3cm}{\label{subfg21}\includegraphics{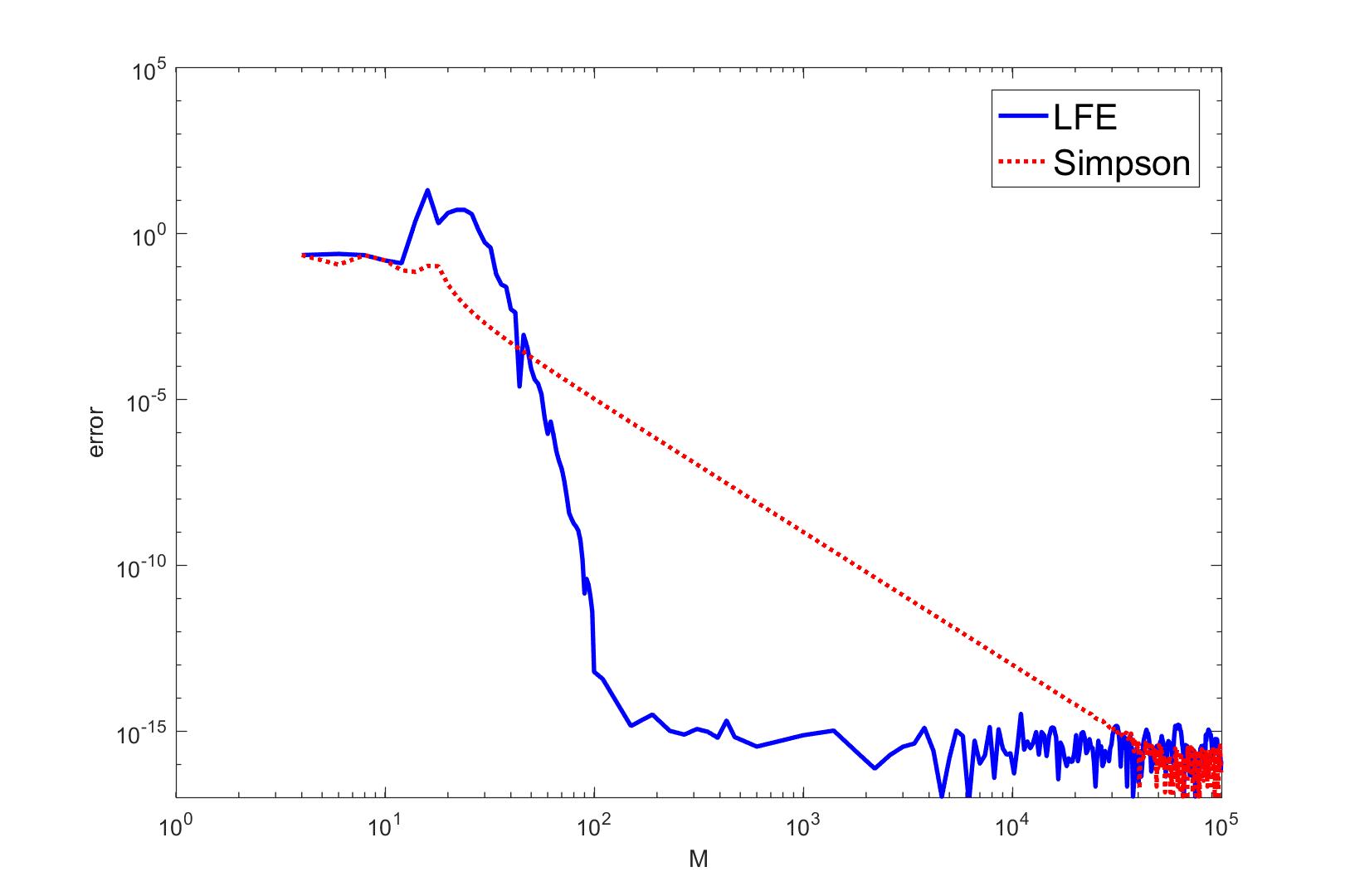}}}
		\subfigure[$\omega=200$]{\resizebox*{6cm}{3cm}{\label{subfg22}\includegraphics{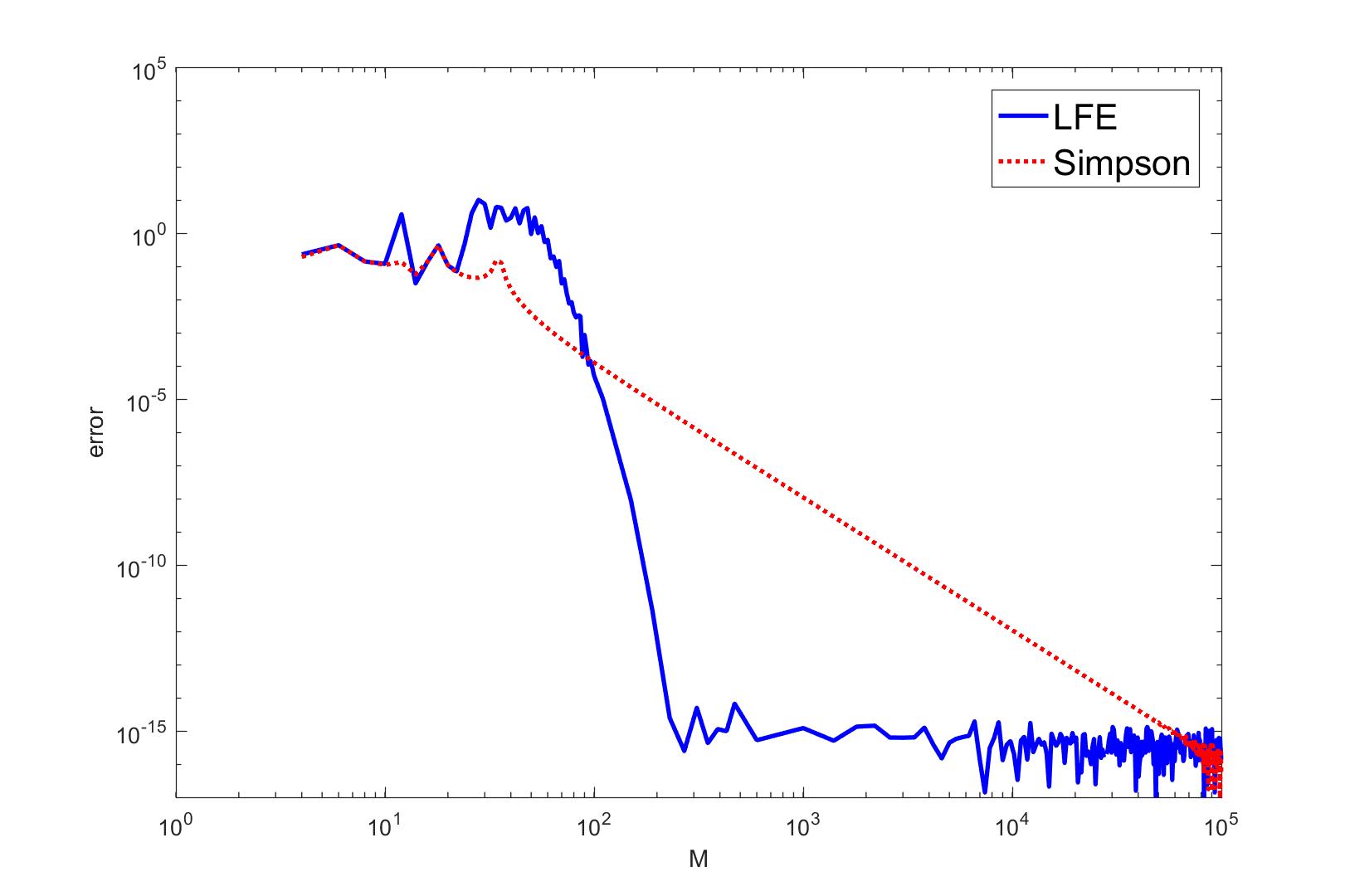}}}
		\caption{Error decay versus the total number of uniform subintervals $M$ for
        the oscillatory test function $f_4(x)=e^{-x}\sin(\omega x)$.}
        \label{fig:osc_f4}
	\end{center}
\end{figure}

\begin{figure}[htbp]
	\begin{center}
		\subfigure[$\kappa=50$]{\resizebox*{6cm}{3cm}{\label{subfg23}\includegraphics{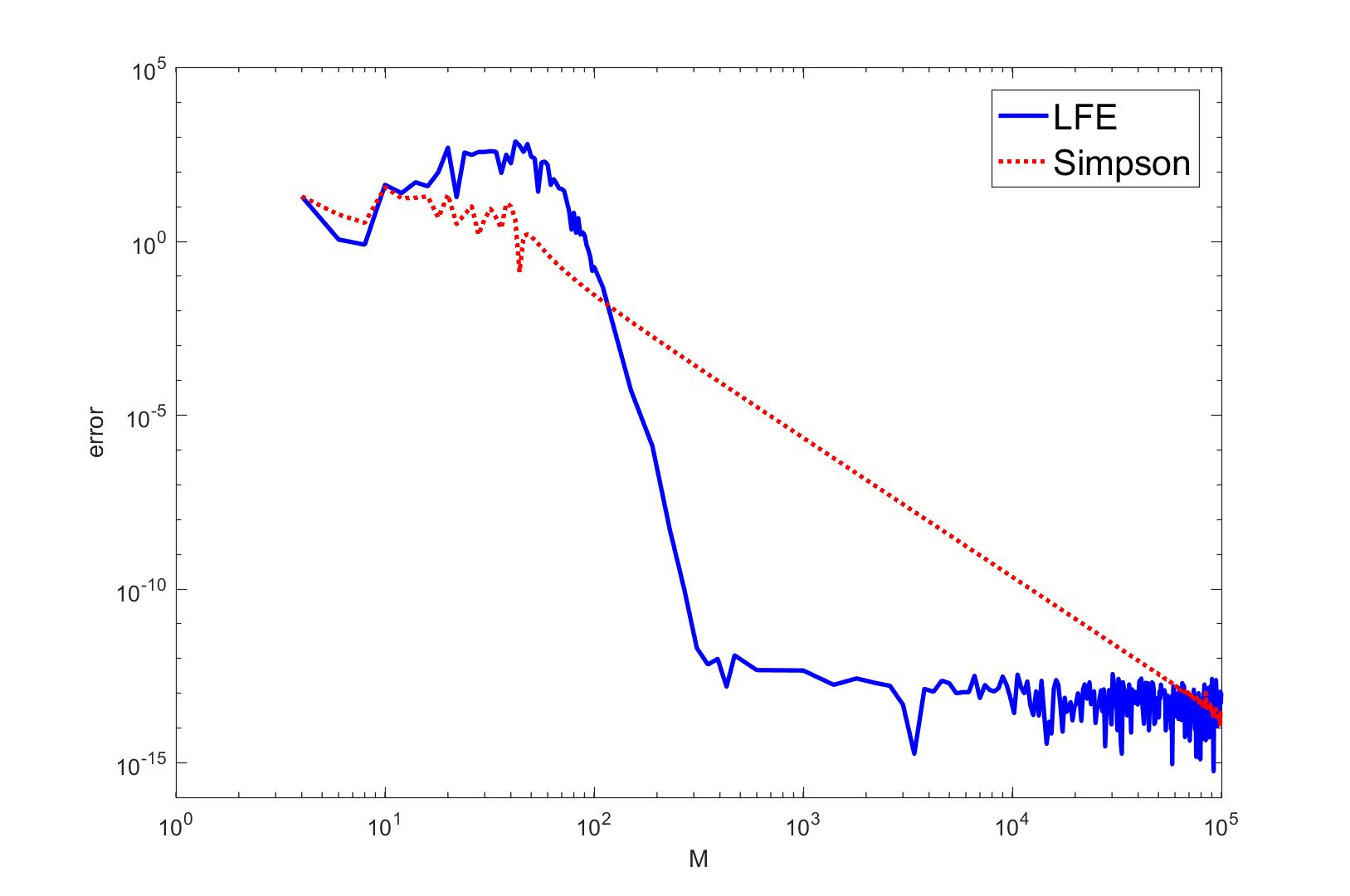}}}
		\subfigure[$\kappa=100$]{\resizebox*{6cm}{3cm}{\label{subfg24}\includegraphics{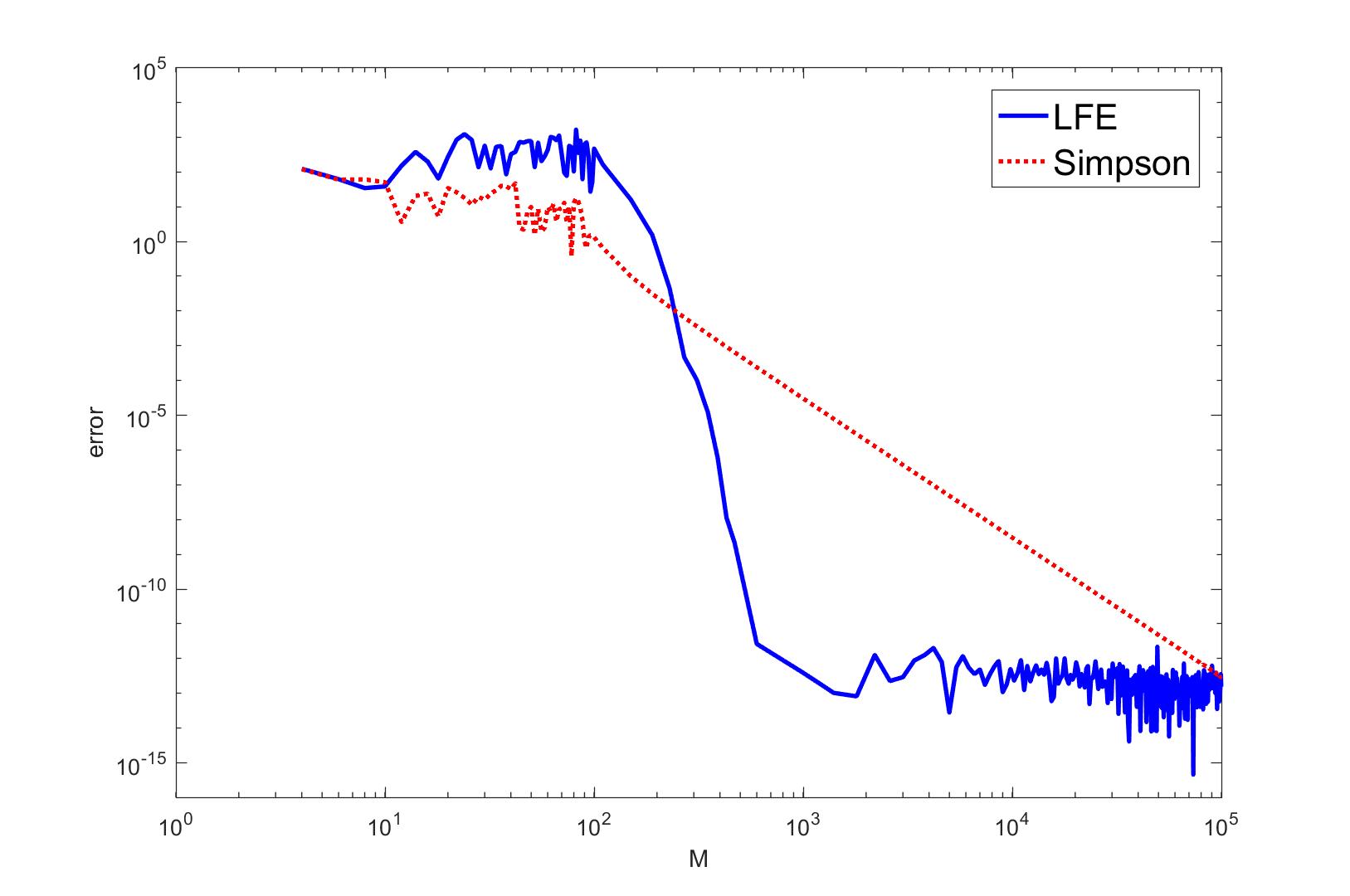}}}
		\caption{Error decay versus the total number of uniform subintervals $M$ for
        the quadratic-phase oscillatory test function $f_5(x)=-2\kappa x\sin(\kappa x^2)$.}
        \label{fig:osc_f5}
	\end{center}
\end{figure}

\begin{figure}[htbp]
	\begin{center}
		\subfigure[$\alpha=0.2$]{\resizebox*{6cm}{3cm}{\label{subfg25}\includegraphics{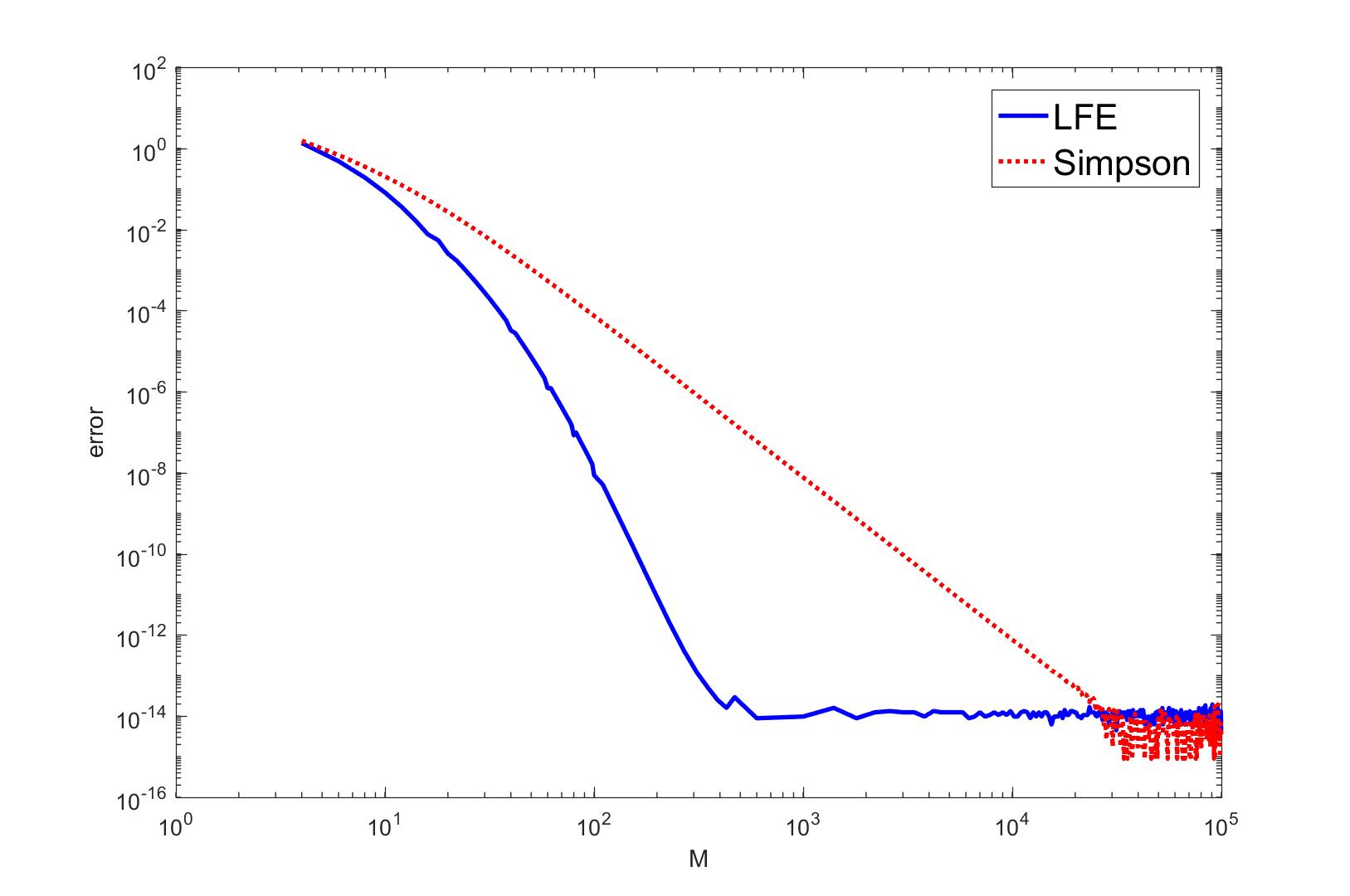}}}
		\subfigure[$\alpha=0.1$]{\resizebox*{6cm}{3cm}{\label{subfg26}\includegraphics{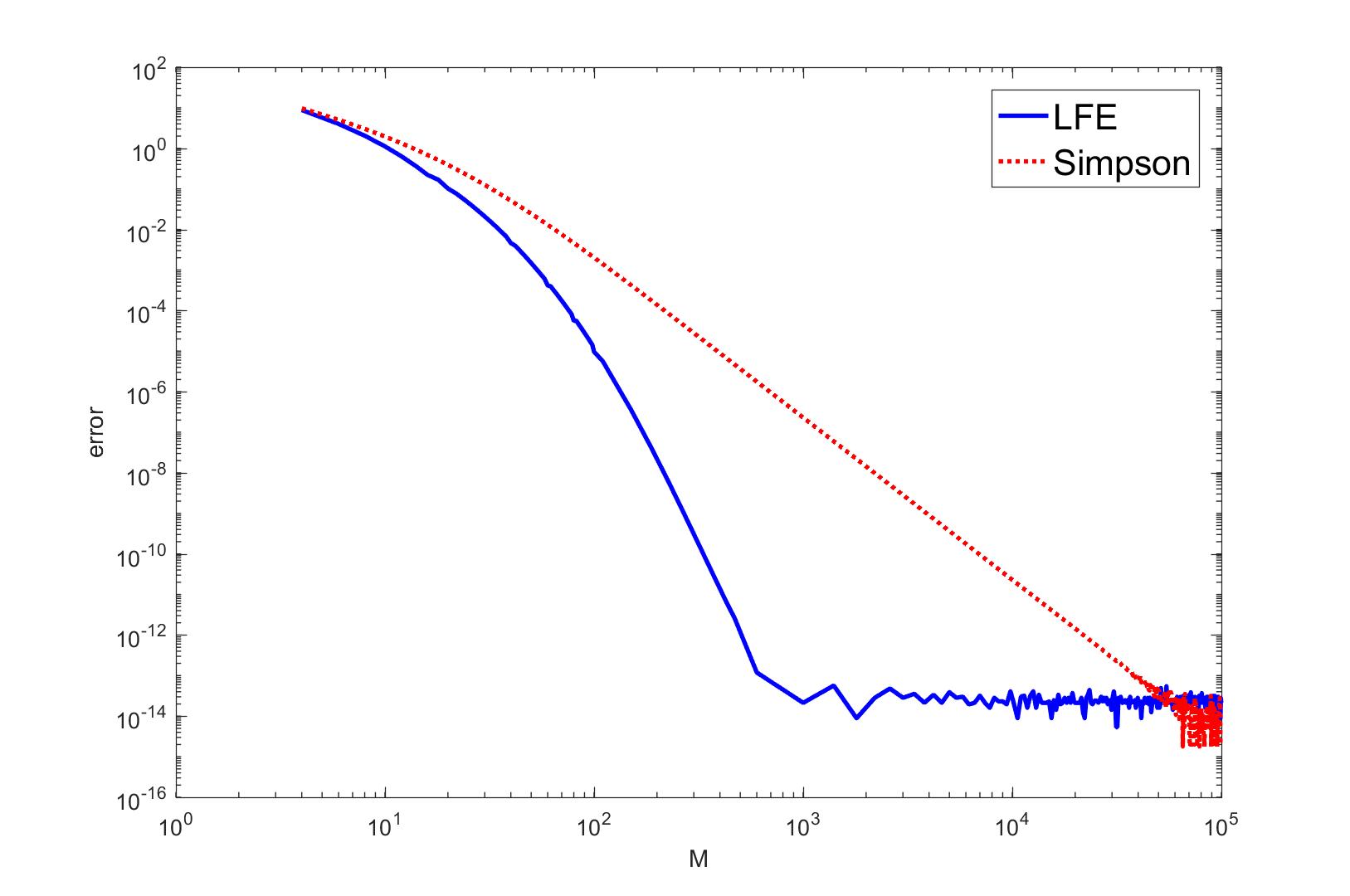}}}
		\caption{Error decay versus the total number of uniform subintervals $M$ for
        the rapidly varying smooth test function
        $f_6(x)=\frac{2x}{(1+\alpha-x^2)^2}$.}
        \label{fig:osc_f6}
	\end{center}
\end{figure}

\subsection{Piecewise smooth integrands}

We finally consider \emph{continuous piecewise smooth} integrands.
In this setting the function itself remains continuous, but its derivative
contains a singular point.
If such a singularity lies inside a local window, the local Fourier
extension approximation is perturbed and the corresponding coefficient
vector becomes abnormally large.
This provides a natural mechanism for singular-window detection.

\paragraph{Test functions.}
We consider two representative piecewise smooth functions
\[
f_7(x)=
\begin{cases}
\dfrac{1}{1+x^2}+\sin(5x), & 0\le x<\xi,\\[2mm]
\dfrac{1}{1+x^2}+\sin(5x)+x-\xi, & \xi\le x\le 1,
\end{cases}
\qquad \xi\in(0,1),
\]
with exact integral
\[
\int_0^1f_7(x)\,dx=\frac{\pi}{4}+\frac{1-\cos 5}{5}+\frac{1}{2}(1-\xi)^2,
\]
and
\[
f_8(x)=
\begin{cases}
e^{x}\cos(2x)+\dfrac{x}{1+x^2}, & 0\le x<\zeta,\\[2mm]
e^{x}\cos(2x)+\dfrac{x}{1+x^2}+(x-\zeta)^2, & \zeta\le x\le 1,
\end{cases}
\qquad \zeta\in(0,1),
\]
with exact integral
\[
\int_0^1f_8(x)\,dx
=
\frac{e\cos2+2\sin2-1}{5}
+\frac{\ln2}{2}
+\frac{(1-\zeta)^3}{3}.
\]

\paragraph{Detection of singularity-containing windows.}
For each local window $I_k$, we compute the coefficient-energy indicator
\[
\eta_k=\|{\bf c}_k^\epsilon\|_2.
\]
Figure~\ref{fig:eta_piecewise} shows the distribution of $\eta_k$ for
several choices of the singular point, with the total number of uniform
subintervals fixed at $M=160$.

Three situations are observed.
If the singular point coincides with a window endpoint
(e.g.\ $\xi=0.5$ or $\zeta=0.25$), then no window contains the singularity
and no abnormal spike appears.
If the singular point coincides with a sampling node but not with a window
endpoint (e.g.\ $\xi=0.3$ or $\zeta=0.6$), then the unique window
containing the singular point is clearly detected by a sharp increase of
$\eta_k$.
If the singular point does not coincide with any sampling node
(e.g.\ $\xi=\pi/5$ or $\zeta=0.73$), the singularity-containing window is
still identified reliably.
In all cases where a singular point lies inside some window, the indicator
$\eta_k$ shows a pronounced spike several orders of magnitude larger than the
values in smooth regions.

\paragraph{Localization of the singular point.}
Once a candidate window is detected, we further localize the singular point
inside that window by examining the coefficient norms
$\|{\bf cL}_i\|_2$ and $\|{\bf cR}_i\|_2$ introduced in
Section~\ref{subsec:edge_refined}.
Table~\ref{tab_loc_piecewise} reports their values for four representative
configurations.

When the singular point coincides with a sampling node
(e.g.\ $\xi=0.3$ and $\zeta=0.6$), both
$\|{\bf cL}_i\|_2$ and $\|{\bf cR}_i\|_2$ remain small near the correct
index.
In contrast, when the singular point lies strictly inside a grid cell
(e.g.\ $\xi=\pi/5$ and $\zeta=0.73$), one of the two quantities remains
moderate while the other becomes very large.
Therefore, the minimum of
\[
\|{\bf cL}_i\|_2+\|{\bf cR}_i\|_2
\]
correctly identifies the smallest grid interval containing the singular
point.

\paragraph{Effect of singularity correction.}
After locating the singular point, the contribution of the affected window is
replaced by the corrected piecewise LFE quadrature described in
Section~\ref{subsec:edge_refined}.
The resulting signed errors are listed in
Table~\ref{tab_piecewise_accuracy}.

Without correction, the presence of a singular point inside a window causes
a visible loss of accuracy, typically at the level of $10^{-4}$--$10^{-6}$
for $f_7$ and between $10^{-7}$ and $10^{-9}$ for the tested cases of
$f_8$.
After applying the singularity correction, however, the error is reduced to
the level of $10^{-15}$, that is, essentially machine precision.
Moreover, this corrected accuracy remains stable as $M$ increases.

These results show that the proposed detection and correction strategy
effectively removes the influence of local singularities and restores the
near-spectral accuracy of the LFE-based quadrature.

\begin{figure}[htbp]
\centering
\subfigure[$f_7,\ \xi=0.3$]{
\includegraphics[width=0.32\textwidth]{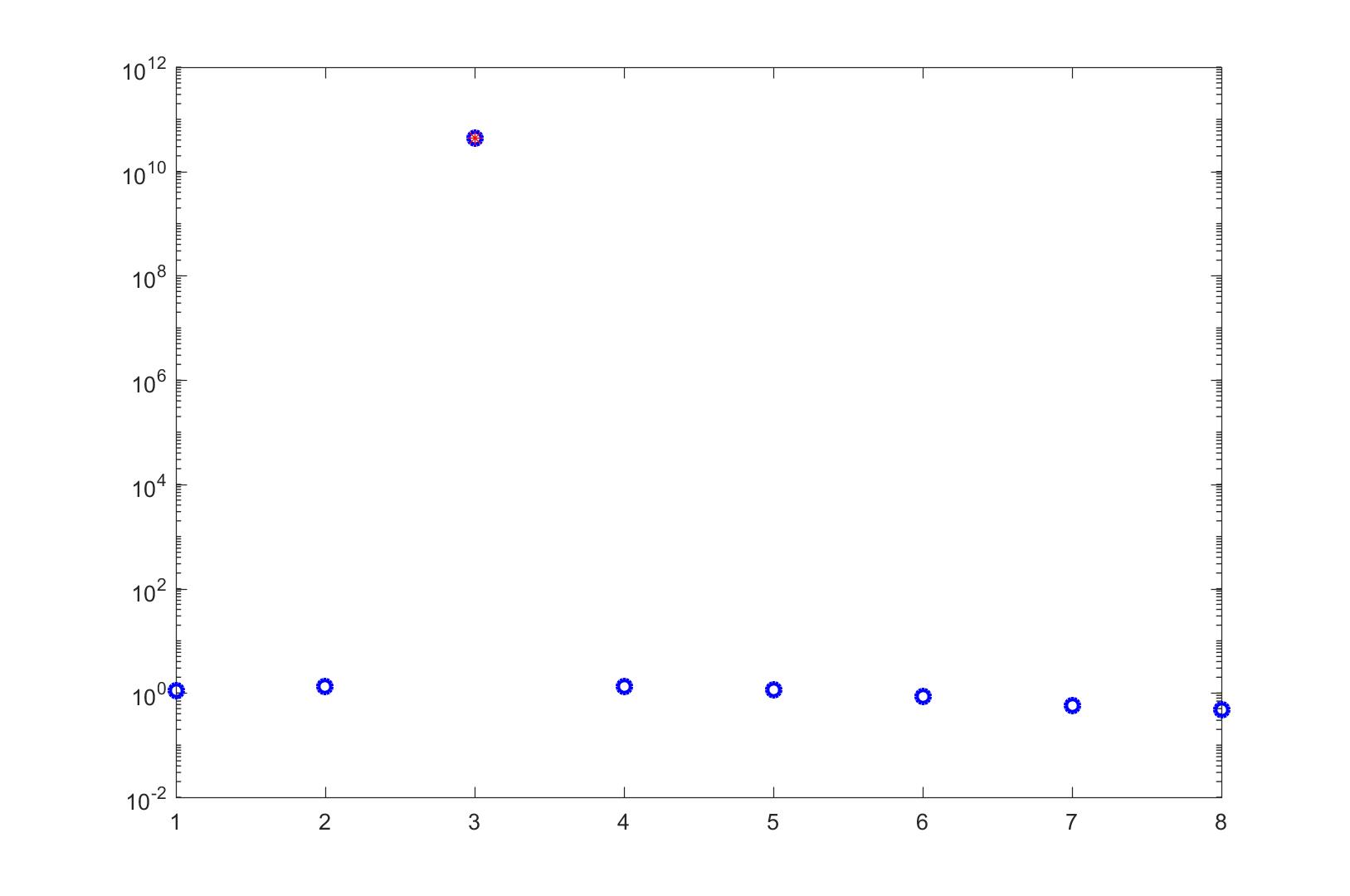}}
\subfigure[$f_7,\ \xi=0.5$]{
\includegraphics[width=0.32\textwidth]{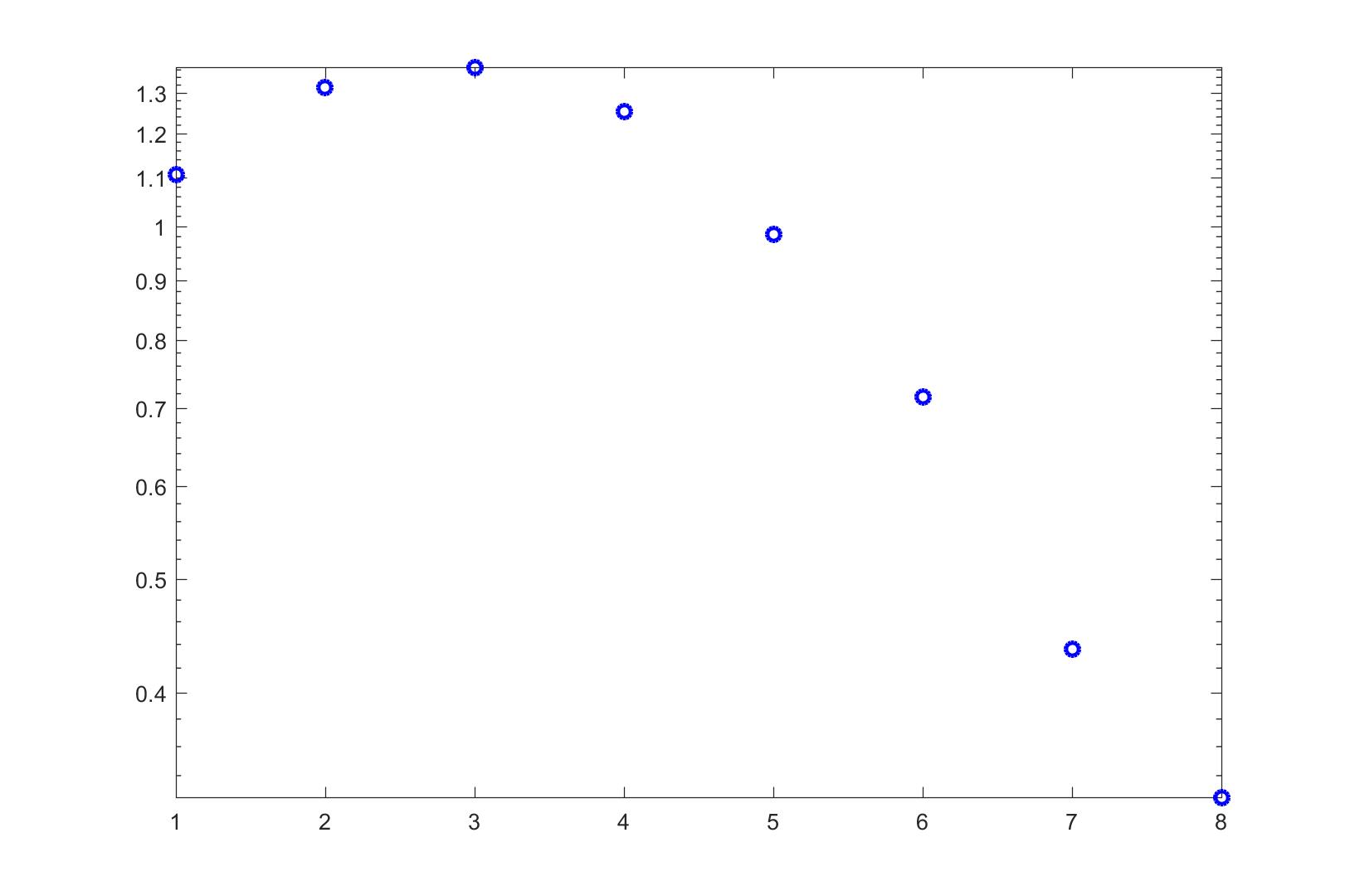}}
\subfigure[$f_7,\ \xi=\frac{\pi}{5}$]{
\includegraphics[width=0.32\textwidth]{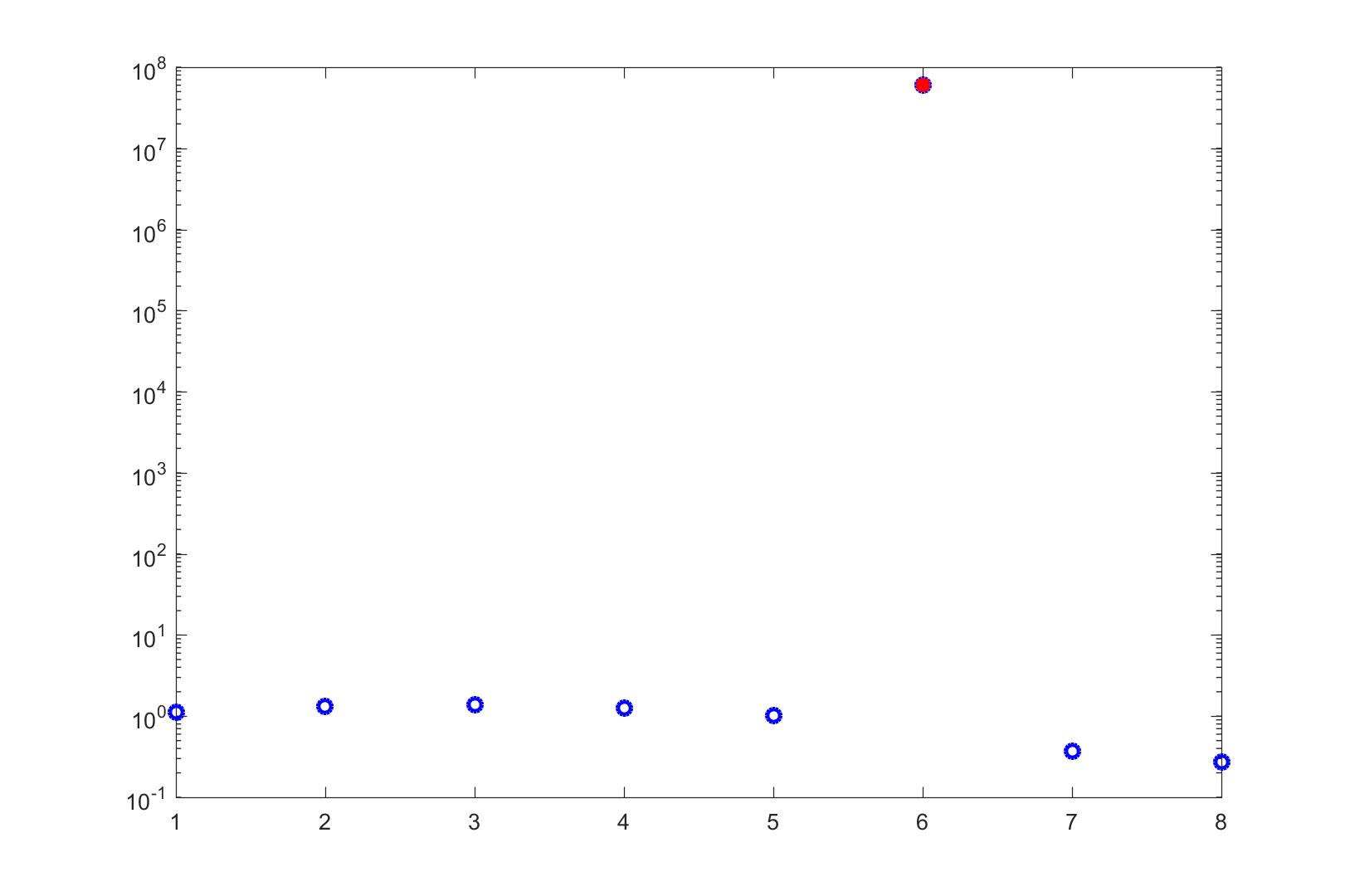}}

\subfigure[$f_8,\ \zeta=0.25$]{
\includegraphics[width=0.32\textwidth]{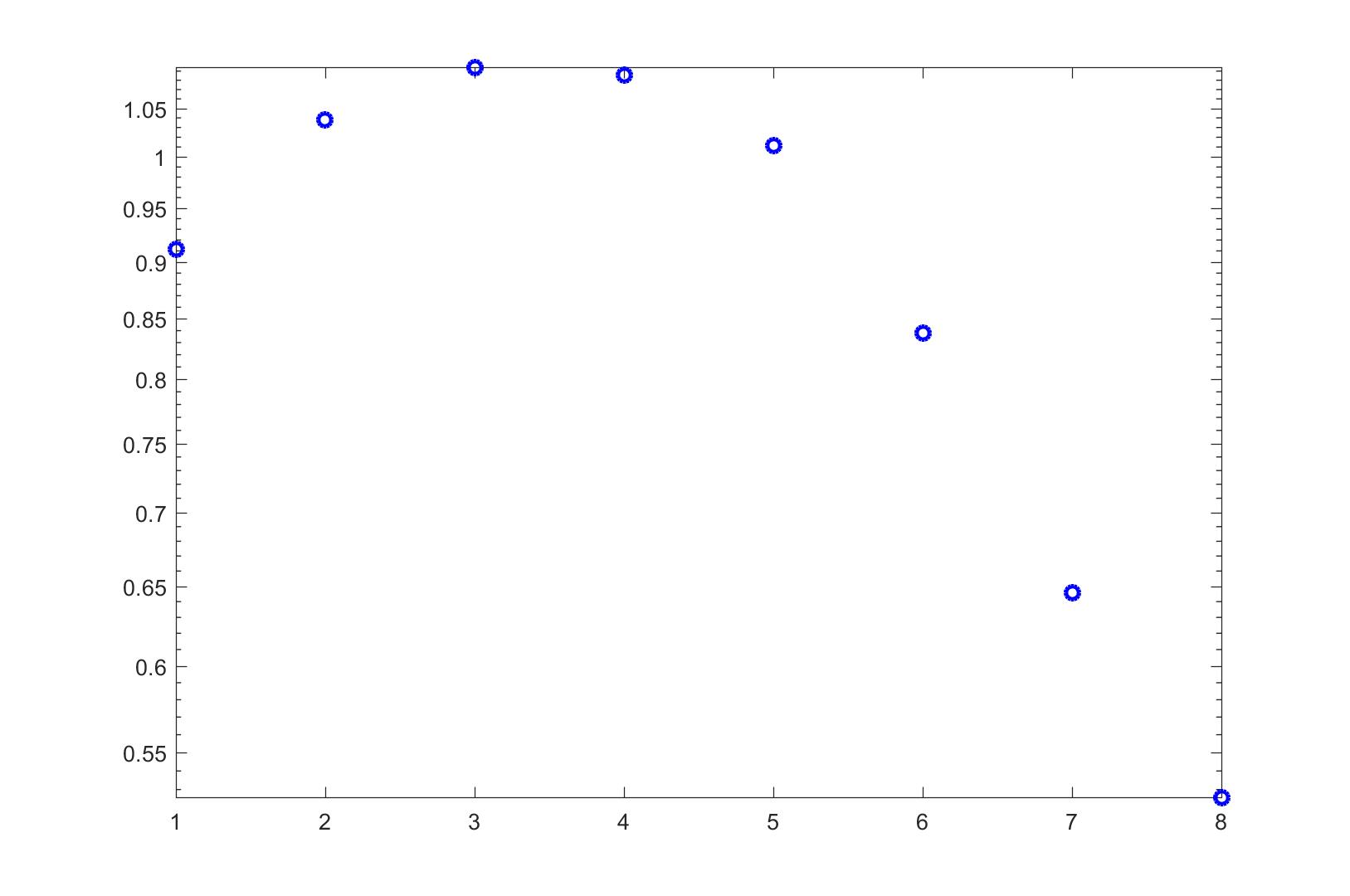}}
\subfigure[$f_8,\ \zeta=0.6$]{
\includegraphics[width=0.32\textwidth]{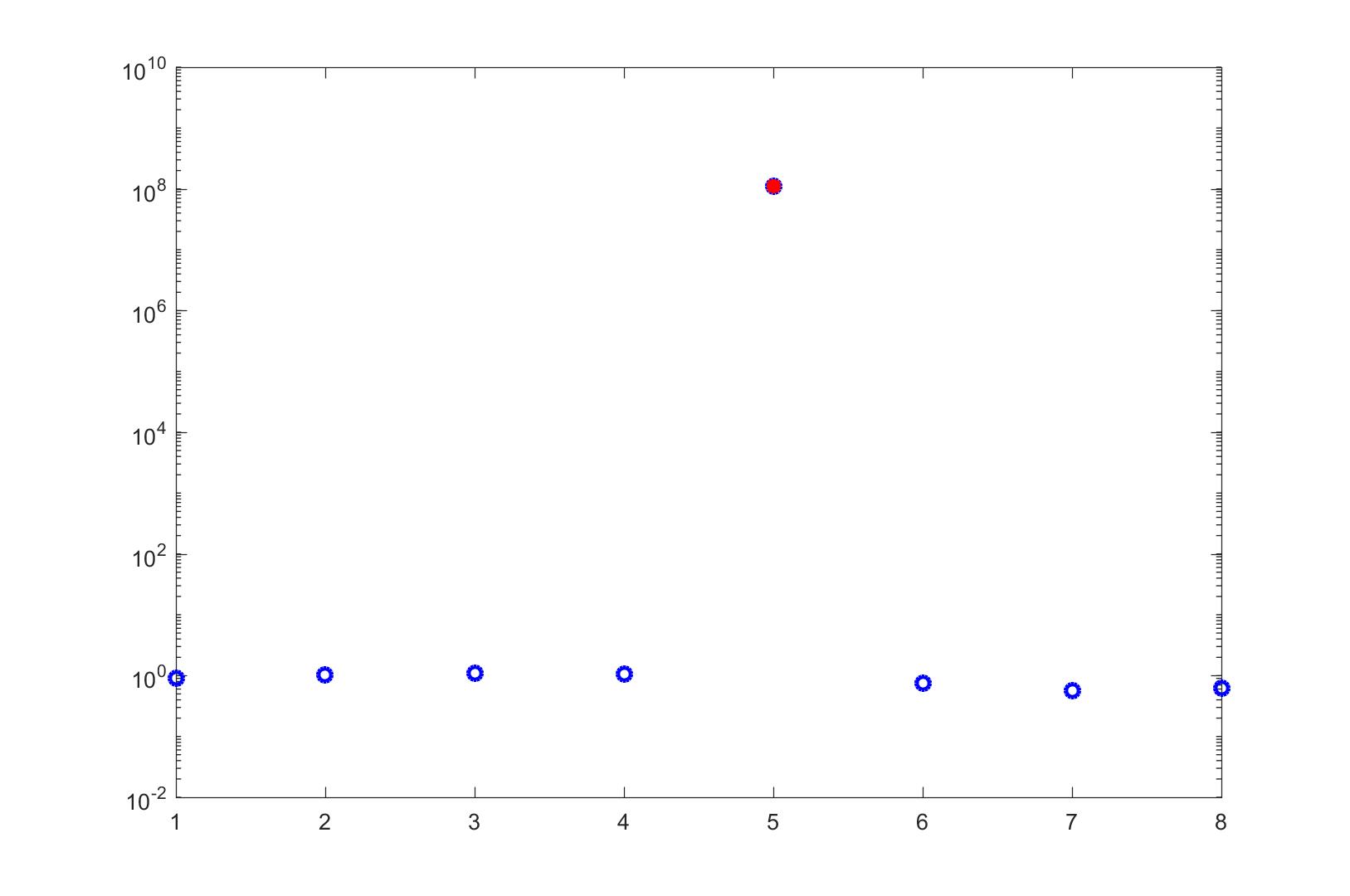}}
\subfigure[$f_8,\ \zeta=0.73$]{
\includegraphics[width=0.32\textwidth]{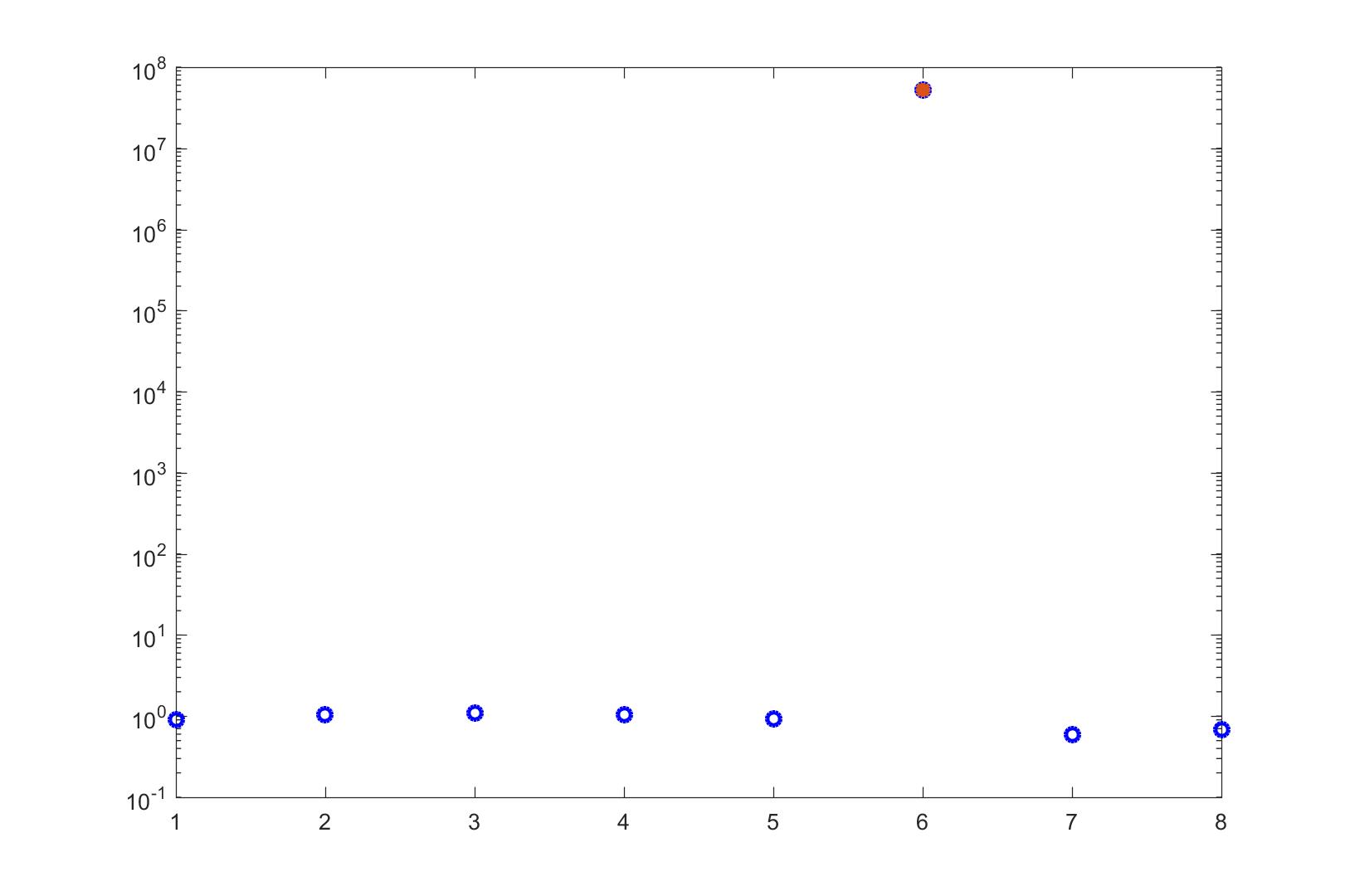}}

\caption{
Distribution of the coefficient energy indicator
$\eta_k=\|{\bf c}_k^\epsilon\|_2$ for the piecewise smooth test functions
$f_7$ and $f_8$ with $M=160$.
Each marker corresponds to one local window $I_k$.
When the singular point lies inside a window
(e.g.\ $\xi=0.3$, $\xi=\pi/5$, $\zeta=0.6$, $\zeta=0.73$),
the corresponding $\eta_k$ exhibits a pronounced spike,
indicating the presence of a singularity in that window.
When the singular point coincides with a window endpoint
(e.g.\ $\xi=0.5$ or $\zeta=0.25$),
no window contains the singularity and therefore the coefficient energy
remains uniformly bounded.
}
\label{fig:eta_piecewise}
\end{figure}

\begin{table}[htbp]
\centering
\caption{
Values of the coefficient norms $\|{\bf cL}_{i}\|_2$ and
$\|{\bf cR}_{i}\|_2$ used for singular point localization inside the
detected window.
The index $i$ denotes the candidate splitting position.
When the singular point coincides with a sampling node
(e.g.\ $\xi=0.3$ and $\zeta=0.6$), both $\|{\bf cL}_i\|_2$ and
$\|{\bf cR}_i\|_2$ remain small near the correct index.
When the singular point lies strictly inside a grid cell
(e.g.\ $\xi=\pi/5$ and $\zeta=0.73$), one of the two quantities remains
moderate while the other becomes extremely large.
Consequently, the minimum of
$\|{\bf cL}_i\|_2+\|{\bf cR}_i\|_2$
correctly identifies the smallest interval containing the singular point.
}
\label{tab_loc_piecewise}
\small
{\begin{tabular*}{\textwidth}{@{\extracolsep\fill}ccccccccccccc}
\toprule
 & \multicolumn{2}{c}{$f_7,\xi=0.3$}  &\multicolumn{2}{c}{$f_7,\xi=\frac{\pi}{5}$} & \multicolumn{2}{c}{$f_8,\zeta=0.6$} & \multicolumn{2}{c}{$f_8,\zeta=0.73$}\\
\cline{2-3} \cline{4-5} \cline{6-7} \cline{8-9}
$i$ & $\|{\bf cL}_{i}\|_2$ & $\|{\bf cR}_{i}\|_2$ & $\|{\bf cL}_{i}\|_2$ & $\|{\bf cR}_{i}\|_2$& $\|{\bf cL}_{i}\|_2$ & $\|{\bf cR}_{i}\|_2$& $\|{\bf cL}_{i}\|_2$ & $\|{\bf cR}_{i}\|_2$\\
\midrule
1&1.18e1&8.13e10&{\color{red}\it 3.46e7}& {\color{blue} \it 0.52e1}&1.02e1&1.20e8&0.89e1&1.07e8\\
2&1.22e1 &8.22e10&4.17e8&0.50e1&1.01e1&8.87e7&0.89e1&2.16e8\\
3&1.21e1 &5.150e10&2.24e9&0.48e1&1.01e1&6.96e7&0.88e1&2.59e8\\
4&1.23e1 & 2.16e10&6.99e9&0.47e1&1.00e1&2.48e8&0.87e1&1.21e8\\
5&1.25e1&5.98e9&1.30e10&0.45e&1.00e1&2.45e8&0.86e1&1.52e8\\
6&1.24e1&9.99e8&  1.23e10&0.43e&1.01e1&1.57e7& 0.86e1& 3.16e8\\
7&1.26e1&7.67e7&3.92e9&0.42e1&0.98e1&2.27e8& 0.84e1&1.87e8\\
8&{\color{blue}\it 1.27e1}&{\color{blue}\it 1.33e1}&2.93e10&0.40e1&0.99e1& 2.50e8&0.84e1&1.11e8\\
9&7.38e7&1.32e1&3.93e10&0.39e1&0.98e1& 8.27e7&0.82e1&2.93e8\\
10&9.72e8&1.34e1&1.68e10&0.37e1&0.97e1&7.73e7& 0.82e1&2.45e8\\
11&5.90e9&1.33e1& 2.11e10&0.36e1&0.97e1&1.15e8&0.81e1&9.49e7\\
12&2.16e10&1.33e1&3.87e10&0.34e1&0.97e1&7.17e7&0.79e1&4.78e6\\
13&5.23e10&1.33e1&2.32e10&0.33e1&0.95e1&2.59e7&0.79e1&2.46e7\\
14& 8.57e10&1.33e1&4.67e9&0.31e1&0.97e1&5.28e6&0.77e1&1.26e7\\
15&8.87e10&1.33e1&1.95e10&0.30e1&0.93e1&4.79e5&0.76e1&3.05e6\\
16&3.80e10&1.31e1&1.71e10&0.28e1&{\color{blue}\it 0.95e1}&{\color{blue}\it 0.76e1}&0.75e1&3.06e5\\
17&4.10e10&1.31e1&8.66e9&0.27e1&4.61e5&0.75e1&{\color{red} \it 1.84e4}& {\color{blue} \it 0.50e1}\\
18&8.23e10&1.31e1&2.71e9&0.25e1&5.16e6&0.74e1&3.84e5&0.49e1\\
19&4.82e10&1.31e1&4.94e8&0.25e1&2.56e7&0.73e1&5.87e6&0.47e1\\
\bottomrule
\end{tabular*}}
\end{table}

\begin{table}[htbp]
\centering
\caption{
Quadrature errors $E$ for the piecewise smooth
test functions $f_7$ and $f_8$ before and after applying the singularity
correction.
The column ``uncorrected'' corresponds to the original LFE quadrature,
while ``corrected'' denotes the result after detecting the singular window
and replacing its contribution by the piecewise LFE integration.
Without correction, the presence of a singular point inside a window causes
a noticeable loss of accuracy.
After applying the singularity correction, the errors are reduced to the
level of machine precision, showing that the proposed procedure restores the
high-order accuracy of the LFE quadrature.
}
\label{tab_piecewise_accuracy}
\small
{\begin{tabular*}{\textwidth}{@{\extracolsep\fill}ccccccccccccc}
\toprule
 & \multicolumn{2}{c}{$f_7,\xi=0.3$}  &\multicolumn{2}{c}{$f_7,\xi=\frac{\pi}{5}$} & \multicolumn{2}{c}{$f_8,\zeta=0.6$} & \multicolumn{2}{c}{$f_8,\zeta=0.73$}\\
\cline{2-3} \cline{4-5} \cline{6-7} \cline{8-9}
$M$ & uncorrected & corrected & uncorrected & corrected & uncorrected & corrected & uncorrected & corrected \\
\midrule
160&1.40e-4&2.89e-15&3.31e-7&4.21e-15&5.56e-7&2.22e-15&1.70e-6&2.44e-15\\
320&5.31e-5&3.77e-15&1.68e-6&2.88e-15&1.91e-7&2.33e-15&1.88e-7&4.44e-15\\
640&1.24e-5&3.11e-15&6.21e-7&5.55e-16&8.47e-9&2.33e-15&2.78e-8&5.66e-15\\
1280&3.25e-6&3.10e-15&9.17e-7&1.11e-15&2.95e-9&2.33e-15&2.52e-9&2.88e-15\\
\bottomrule
\end{tabular*}}
\end{table}

\subsection{A brief comparison with Clenshaw--Curtis quadrature}\label{SEC4_CC}

To further illustrate the behavior of the proposed method, we include a brief comparison with the standard Clenshaw--Curtis (CC) quadrature under a fixed sampling budget.
The CC rule is evaluated at its Chebyshev-type (nonuniform) nodes, while the proposed method is designed for integration from uniformly sampled data.
For clarity and reproducibility, we report the errors obtained with the same number of sampling points, rather than using adaptive refinement for CC.
In this comparison, we only include the test functions $f_4$, $f_5$, $f_7$, and $f_8$.
For the simpler cases $f_1$--$f_3$, both methods reach near machine precision with fewer than $30$ nodes, and thus the comparison is less informative.
To make the benchmark more representative, we slightly adjust several parameters (namely $\omega$, $\kappa$, $\xi$, and $\zeta$) in these selected functions, while keeping the integration interval unchanged.

Table~\ref{tab:cc_compare} reports the absolute errors of the proposed LFE-based quadrature and the CC rule for the above four functions.
For the highly oscillatory examples $f_4$ and $f_5$, CC achieves very high accuracy already at moderate sample sizes, whereas the proposed method requires a larger sampling budget to attain comparable precision.
In particular, for $f_4$ with $\omega=200$, CC is close to machine precision already at $M=256$, while the proposed method reaches machine precision at $M=512$.
For $f_5$ with $\kappa=100$, CC also converges rapidly, while the proposed method exhibits a slower decay for small $M$ but eventually attains high accuracy as $M$ increases.

For the non-smooth examples $f_7$ and $f_8$, which are continuous but have localized derivative singularities, the situation is different.
The convergence of CC becomes significantly slower: even at $M=1024$, the CC errors remain around $10^{-8}$ for $f_7$ and around $10^{-11}$ for $f_8$.
By contrast, the proposed method maintains near machine precision already at $M=128$ for both $f_7$ and $f_8$.
This clear gap reflects the main advantage of the proposed framework on uniform grids: the local detection-and-correction mechanism can effectively isolate the influence of localized nonsmoothness and correct only the affected contributions.

Overall, the above results indicate that the two methods have complementary strengths.
Clenshaw--Curtis quadrature is highly effective for globally smooth functions (including highly oscillatory ones) when nonuniform quadrature nodes are allowed.
In contrast, the proposed method is particularly advantageous for continuous piecewise smooth integrands with localized derivative singularities in the uniform-grid setting considered in this paper.

\begin{table}[htbp]
\centering
\caption{A comparison between the proposed LFE-based quadrature (LFE) and the Clenshaw--Curtis quadrature (CC) under the same sampling budget.
The table reports the absolute errors for $f_4$, $f_5$, $f_7$, and $f_8$ with parameters $\omega=200$, $\kappa=100$, $\xi=0.3$, and $\zeta=0.73$, respectively.
The integration interval is the same as in the previous experiments.}
\label{tab:cc_compare}
\small
{\begin{tabular*}{\textwidth}{@{\extracolsep\fill}ccccccccccccc}
\toprule
 & \multicolumn{2}{c}{$f_4,\omega=200$}  &\multicolumn{2}{c}{$f_5,\kappa=100$} & \multicolumn{2}{c}{$f_7,\xi=0.3$} & \multicolumn{2}{c}{$f_8,\zeta=0.73$}\\
\cline{2-3} \cline{4-5} \cline{6-7} \cline{8-9}
$M$ & LFE & CC & LFE & CC& LFE & CC& LFE & CC \\
\midrule
128&0.20e-1 &   1.80e-11&2.68e1&1.18e-10&2.22e-16&8.04e-7&3.10e-15&2.89e-9\\
256& 9.80e-8 &6.25e-16&2.02e-3&1.93e-13&1.93e-14&1.30e-6&7.77e-15&1.158e-9\\
512&2.71e-15 &5.10e-16&4.34e-11&4.73e-14&6.66e-15& 4.12e-7&4.99e-15&2.41e-10\\
1024&1.79e-16&1.28e-16& 2.82e-13&1.14e-13&1.99e-15&5.03e-8&2.44e-15& 4.03e-11\\
\bottomrule
\end{tabular*}}
\end{table}
\section{Conclusions and Remarks}\label{SEC5}

In this paper, we proposed a high-precision numerical quadrature method based on local Fourier extension (LFE).
The integrand is approximated on overlapping uniform windows by stabilized local FE expansions, and the resulting integrals are evaluated analytically from the local Fourier coefficients.
This establishes a direct link between local spectral reconstruction and numerical quadrature, yielding an integration framework that combines spectral accuracy with the practical convenience of uniform sampling.

\begin{itemize}
  \item \textbf{Uniform-grid setting.}
  A key feature is that the method works directly with equispaced samples.
  Unlike Gaussian-type quadratures that require specially distributed nodes, the proposed scheme can be implemented using uniformly sampled data, which is particularly attractive when function values are only available on fixed grids (e.g., from measurements or from numerical simulations).

  \item \textbf{Error control and accuracy.}
  We derived an error estimate showing that the quadrature error is controlled by the LFE approximation error, and hence the quadrature inherits the approximation properties of the local FE reconstruction.
  In particular, for smooth functions the method achieves high accuracy with substantially fewer nodes than classical composite rules on uniform grids.

  \item \textbf{Numerical validation and robustness.}
  The experiments confirm the above advantages.
  For globally smooth test functions, the proposed quadrature attains near-machine precision with significantly fewer nodes than the composite Simpson rule.
  For oscillatory and variable-frequency integrands, the advantage becomes more pronounced due to local rescaling on each window.
  For continuous piecewise smooth integrands, we further introduced a singular-point detection and correction strategy based on coefficient-energy indicators, local singularity localization, one-sided value prediction, and local integral replacement; this procedure effectively removes the deterioration caused by localized derivative singularities and restores near-spectral accuracy.

  \item \textbf{Adaptive extension.}
  Although we adopt uniform sampling throughout, the framework can be naturally combined with an \emph{a posteriori} local approximation indicator to design an adaptive sampling strategy, refining only where the local error is large.
  This may further reduce the total number of nodes for integrands with strongly nonuniform frequency distributions or localized rapid variation.
  We leave a systematic study of such adaptive refinements to future work.
\end{itemize}

Several directions deserve further investigation:
(i) developing a rigorous adaptive refinement criterion and analyzing its convergence;
(ii) extending the singularity treatment from continuous piecewise smooth functions to genuine jump discontinuities, which requires additional one-sided information;
and (iii) generalizing the current framework to multidimensional settings (e.g., tensor-product constructions) and to integrals arising in more general numerical PDE and inverse-problem applications.
\section*{References}

\end{document}